\documentclass[12pt]{article}
\usepackage{amsfonts,amsmath,amsxtra,amsthm,amssymb,latexsym}
\usepackage{bbm}

\newtheorem{theorem}{Theorem}[section]
\newtheorem{corollary}[theorem]{Corollary}

\newtheorem{proposition}[theorem]{Proposition}

\newtheorem{definition}[theorem]{Definition}
\newtheorem{remark}[theorem]{Remark}
\newtheorem{problem}{Problem}
\theoremstyle{remark}
\newtheorem{example}[theorem]{Example}

\numberwithin{equation}{section}

\DeclareMathOperator{\dom}{Dom}
\DeclareMathOperator{\col}{col}

\def\R{\mathbb R}
\def\C{\mathbb C}
\def\N{\mathbb N}
\def\Z{\mathbb Z}

\def\wt{\widetilde}

\def\vphi{\varphi}
\def\ga{\gamma}

\def\X{\mathcal{X}}
\def\B{\mathcal{B}}
\def\l{\ell}

\def\L{\mathcal{L}}
\def\W{\mathcal{W}}
\def\V{\mathcal{U}}
\def\I{I}
\def\Pr{P}
\def\Ssp{\mathcal{S}}
\def\vz{0}
\def\Sh{S}
\def\A{A}
\def\E{E}

\def\h{h}
\def\g{\mathfrak{g}}
\def\e{\mathfrak{e}}
\def\fz{\mathbf{0}}
\def\s{s}
\def\fin{\mathrm{fin}}
\def\subd{\mathrm{subd}}

\begin{document}
\thispagestyle{empty}
%\markboth{}{}
\title{Bohl-Perron type stability theorems for linear
difference equations with infinite delay}

\author{Elena Braverman and Illya M. Karabash}

%\date{\today}

\maketitle

\begin{abstract}
Relation between two properties of
linear difference equations with infinite delay is investigated: (i)
exponential stability, (ii) $\l^p$-input $\l^q$-state stability 
(sometimes is called Perron's property). The latter means that
solutions of the non-homogeneous equation with zero initial data
belong to $\l^q$ when non-homogeneous terms are in $\l^p$.
%(for $p=q=\infty$ it is called Perron's property).
It is assumed that at each moment the prehistory (the sequence of preceding states) belongs
to some weighted $\l^r$-space with an exponentially fading weight (the phase space).

Our main result states that (i) $\Leftrightarrow$ (ii) whenever $(p,q) \neq (1,\infty)$
and a certain boundedness condition on coefficients is fulfilled. This condition is sharp and
ensures that, to some extent, exponential and $\l^p$-input $\l^q$-state stabilities does not depend 
 on the choice of a phase space and parameters $p$ and $q$, respectively.
$\l^1$-input $\l^\infty$-state stability corresponds to uniform stability.
We provide some evidence that similar criteria should not be expected for non-fading memory spaces.
\end{abstract}

%\footnotetext[2]{Partially supported by the NSERC Research Grant and
%the AIF Research Grant}

\vspace{2mm}

\vspace{2mm}

\section{Introduction}

We consider systems of linear difference equation with infinite delay
\begin{equation}
x(n+1) = L (n) x_n + f(n),  \quad n \geq 0, \label{e nh}
\end{equation}
which in particular include \emph{Volterra difference systems}
\begin{equation}
\label{intro1}
x(n+1)=\sum_{k=-\infty}^{n} \ L(n,n-k) \ x(k) \ + f(n), \quad n\geq 0.
\end{equation}
It is assumed that $x(\cdot)$ is a discrete function from $\Z$ to a (real or complex) Banach space
$\X$, $f(\cdot)$ is a function from $\Z^+ (= \N \cup \{0\})$ to $\X$. The notation $| \cdot |$ 
stands for the norm in $\X$.

By $x_n$ we denote the semi-infinite prehistory sequence 
$ \left\{ x(n) , x(n-1), \dots, x (n+m), \dots \right\}$, $m \leq 0$. 
We suppose that the initial conditions, i.e., the sequence  $x_0 = \{ x(n+m)\}_{m=-\infty}^{0}$, 
belongs to an exponentially weighted $\l^\infty$-space $\B^\ga$, which is called
the phase space. More precisely, it is assumed that for certain $\ga \in \R$
\begin{equation*}
 | x_0 |_{\B^\ga}:=\sup_{m \leq 0} |x(m)| e^{\ga m} < \infty  \quad
\end{equation*}
%(the latter evidently implies $x_n \in \B^\ga$ for all $n$)
and that $L(n)$, $n \geq 0$, are bounded linear mappings from $\B^\ga$ to $\X$.

The aim of the paper is to study relations between uniform exponential stability, uniform stability, 
and $\l^p$-input $\l^q$-state stability (or shorter $(\l^p,\l^q)$-stability) of (\ref{e nh}). 
The precise definitions are given in Section \ref{ss sol}.

%For the Volterra difference system (\ref{intro1}), the last assumption means that
%operators $L(n,k)$ are bounded in $\X$ and
%\begin{equation}
%\label{e Lnk inBga}
%\sum_{k=0}^{+\infty} e^{k \gamma} \| L(n,k) \|_{\X \to \X} < \infty \quad \text{for all}
%\quad n \geq 0 .
%\end{equation}

%\begin{trem}
%Do we want to write somewhere here '(see details in Section \ref{ss sp})' or
%'see e.g. \cite{M97,MM05,CC09}' in connection with phase spaces?
%\end{trem}

For ordinary differential equations with coefficient
$a(t)$ satisfying
\begin{equation}
\int_t^{t+1} | a(\tau)| \, d\tau \leq M_1 < \infty ,\label{e intbound}
\end{equation}
boundedness of a solution of the initial value problem
\begin{equation}
\label{DE}
x^{\prime} (t)+a(t) x(t)=f(t),  \qquad x(0)=0,  \qquad t \geq 0,
\end{equation}
for any bounded on $[0, \infty)$ right hand side $f$ implies exponential stability
of the corresponding homogeneous equation
$ x^{\prime} (t)+a(t) x(t)=0$.
% if $\sup_{t \geq 0} a(t) < \infty$.
This result goes back to Bohl \cite{B14} and then was reinvented by Perron \cite{P30}; the
above relation is sometimes called the Perron property. 
The Bohl-Perron result was extended to arbitrary Banach phase spaces by M. Krein, 
see notes to \cite[Chapter III]{DK74}.
The result was later generalized in the following two directions.
On the one hand,  ``for any $f\in {\bf L}^{\infty}$ the solution
$x\in {\bf L}^{\infty}$" can be substituted by ``for any $f\in \mathfrak{B}_1$
the solution $x\in \mathfrak{B}_2$", where $\mathfrak{B}_1$ and $\mathfrak{B}_2$
are some Banach spaces of functions (see e.g. \cite[Problems III.10-16]{DK74} and references therein). 
In the terminology of our paper this property
is called $(\mathfrak{B}_1, \mathfrak{B}_2)$-\emph{stability}.
On the other hand, the Perron property was studied for various equations,
such as delay differential equations, impulsive delay differential
equations and difference equations.
For first order system of difference equations
%\begin{eqnarray}
$z(n+1) = \A (n) z (n) + \s (n)$, %\quad
$ n  \geq 0$,
%\label{e firstord}
%\end{eqnarray}
%in a Banach space $\W$,
the relation between $(\l^p,\l^q)$-stability
and exponential stability was considered in \cite{P87,P88}
(according to \cite{P87}, this theory goes back to \cite{CS67}) and later in
\cite{AVM96,AVMZ94,NN05,P04,SS04} (see Theorem \ref{t ordUES} below and subsequent remarks).

The Perron property for higher order difference equations %(with bounded and unbounded delay)
was studied in \cite{BB_FDE04,BB_JMAA05,BB_UnbDel,CKRV00}. The paper \cite{CKRV00} deals with
\emph{Volterra difference systems with unbounded delay}
\begin{equation}
\label{e Volunb}
x(n+1)=\sum_{k=0}^{n} \ L(n,n-k) \ x(k) \ + f(n), \qquad n\geq 0;
\end{equation}
exponential stability is understood in the sense of uniform in $n$ estimates on the fundamental 
(resolvent) matrix, the role of the spaces $\mathfrak{B}_{1,2}$ 
is played by exponentially weighted $\l^1$ and $\l^\infty$ spaces. 
In \cite{BB_FDE04,BB_JMAA05,BB_UnbDel}, $(\l^p,\l^q)$-stability for usual (nonweighted) $\l^p$-spaces 
is considered, estimates on the fundamental matrix 
are obtained and then applied to stability and exponential stability of equations with 
finite prehistory $ \{ x(n) \}_{n=-N}^0$, $N < \infty$. The case of bounded delay and 
$1\leq p=q \leq \infty$  is considered in \cite{BB_FDE04,BB_JMAA05} 
(as well as the case $\mathfrak{B}_{1}= \mathfrak{B}_{2}=\mathrm{c}_0$), 
the unbounded delay and $p=q=\infty$ in \cite{BB_UnbDel}.

The problem of finding Bohl-Perron type stability criteria for difference systems with infinite
delay naturally requires the phase space settings of \cite{MM04,MM05,M97} (see Section \ref{ss sp}).
\emph{In the present paper we solve this problem in exponentially fading phase spaces $\B^\ga$, $\ga >0$.} 
Following the differential version of \cite[Chapter 7]{HMN91}, we use the notion of 
\emph{uniform exponential stability in $B^\ga$}, 
which is different from stability properties considered in \cite{BB_FDE04,BB_JMAA05,BB_UnbDel,CKRV00} 
(and is more appropriate for the infinite delay case).
The method is based on the reduction of %a difference system with infinite memory
(\ref{e nh})
to a first order system with states in the phase space.
For systems with bounded delay this method has been announced in \cite{BB_UnbDel}.
The main difficulty is the fact that the $(\l^p,\l^q)$-stability property of (\ref{e nh}) 
is weaker than that of the reduced first order system.

The main points of the present paper are:
\begin{enumerate}
\item Uniform exponential stability and uniform stability are characterized in terms of 
$(\l^p,\l^q)$-stability (Theorems \ref{t UES} and \ref{t US}).
For the particular case of Volterra difference systems these results can be written in the following way. 
\emph{Let $\gamma>0$. Assume that either $p\neq 1$ or $q \neq \infty$. Then the homogeneous
system associated with (\ref{intro1}) is uniformly exponentially
stable in $\B^\ga$ if and only if system (\ref{intro1}) is $(\l^p,\l^q)$-stable and%}
\begin{equation}
\label{intro3}
\text{there exists a positive integer} \ l  \ \text{such that} \quad 
\sup_{n \geq 0} \ \sum_{k \geq l} \ e^{k \gamma} \| L(n,k) \|_{\X \to \X} < \infty.
\end{equation}
%\emph{
The homogeneous
system associated with (\ref{intro1}) is uniformly
stable in $\B^\ga$ if and only if (\ref{intro1}) is $(\l^1,\l^\infty)$-stable and (\ref{intro3}) holds.}

\item \emph{It is an immediate corollary that for systems with bounded delay condition (\ref{intro3}) 
can be omitted.}
    %(see Corollarys \ref{c UESbd} and \ref{c USbd}).}
    This is an exact analogue of the case of a first order difference system 
    for which Bohl-Perron type criteria do not require any boundedness restrictions on the coefficients 
    $\A (n)$ (see \cite{P87,P88,SS04} and Section
    \ref{ss sol} for details).
For differential equation (\ref{DE}) assumption (\ref{e intbound})
cannot be omitted: %in the Bohl-Perron criterion:
generally, the boundedness of a solution for any bounded right hand side does not imply 
exponential stability (see \cite[Section III.5.3]{DK74}). %p. 131
Similarly, in the case of systems with infinite delay,  some assumptions 
involving uniform boundedness of coefficients are essential in the Bohl-Perron criteria.
\emph{Remark \ref{r exPmn} and Example \ref{ex Pmn} shows that, in some sense, (\ref{intro3}) 
is the weakest possible assumption of such type.}

%The requirement of uniform boundedness of linear operators is omitted in citations
%when necessary. However, for equations with infinite delay we demonstrate that
%an additional condition (see (\ref{intro3})) is necessary for the equivalence
%of the Perron condition and the exponential stability. As a corollary, we deduce
%that this condition can be omitted for systems with bounded delay and modified
%for equations with a finite, but unbounded delay. Examples in Section 6
%illustrate sharpness of introduced conditions.

\item There are two other interesting corollaries. Namely, \emph{under the condition of (\ref{intro3})
or its more general version (\ref{e LPlinf}), (i) $(\l^p,\l^q)$-stability does not depend on $p$ and $q$ 
(excluding the case $(p,q) = (1,\infty)$), 
(ii) exponential stability in $\B^\delta$ does not depend on the choice of $\delta \in (0,\ga]$.}
Examples of Section \ref{s ex} 
show that for systems with unbounded delay these statements are not valid without condition (\ref{intro3}).

\item It is essential that we consider exponentially fading phase spaces $\B^\ga$, $\ga >0$. 
Example \ref{ex BB05 ex2} shows that the main results 
(Theorems \ref{t UES} and \ref{t US}) are not valid in the non-fading phase space $\B^0$. 
Nevertheless, for uniform stability in $\B^0$ we give two sufficient conditions of Bohl-Perron type 
(Corollary \ref{c UEga=0}).

\item Main results can be easily extended to exponentially fading phase spaces of $\l^p$ type
(see discussion in Section \ref{s Disc}).
 \end{enumerate}

The paper is organized as follows. After introducing in Section \ref{s Prelim} some notations and
presenting known results (which will be required in the sequel),
we formulate the criterion of uniform exponential stability (Theorem \ref{t UES}) 
and some of its corollaries in Section \ref{ss UESres}. The proof of Theorem \ref{t UES} is postponed 
to Section \ref{ss proofUES} and is preceded by auxiliary propositions of Section \ref{s RedOrd}, which
constitute the technical core of our method.
%the equivalence of the Perron condition and the uniform exponential stability, under an additional restriction)
In Sections \ref{ss preUS}-\ref{ss USres ga>0} the above scheme is mimicked for uniform and 
$(\l^1,\l^\infty)$ stabilities in $\B^\ga$ with $\ga>0$. Section \ref{ss USres ga=0} is devoted to 
the more difficult question of uniform stability in $\B^0$.
%and  stabilitydeals with the case when any solution is bounded for any right hand side in .
The independence of exponential stability of the choice of a phase space is discussed in
Section \ref{s Ind ga} with the use of the notion of \emph{subdiagonal systems} (subdiagonal systems
are equivalent to shifted Volterra systems with unbounded delay, however the shift affects essentially 
the property of $(\l^p,\l^q)$-stability). Section 6 involves all relevant examples
demonstrating sharpness of theorems' conditions.  Section 7 contains discussion and open problems,
as well as some additional applications of the presented method.

%It is known that in the case of a finite delay
%this implies uniform, not exponential stability \cite{AVM96,BB_JMAA05}.
%We prove the same result, under some restrictions on the phase space of
%initial conditions and demonstrate that, generally, it is incorrect for infinite
%delays in Section 4. However, under some additional restrictions boundedness
%of solutions for any right hand side in $\l^1$ implies uniform stability
%(Section 4.3). Section 5 demonstrates that stability properties are in
%some sense independent of the choice of the phase space, once it is
%exponentially fading.

Finally note that other aspects of stability and boundedness of difference systems with unbounded delay
 were studied e.g.
in \cite{CC09,CKRV00,E05,EMV07,EM96,KC-VT-M03,KS2003,LF2002,MM04,M97,Pituk2003}. 
An extensive list of applications can be found e.g. in \cite{KC-VT-M03}.
Most of these papers are devoted to Volterra difference
systems with unbounded delay.

\section{Preliminaries and notation \label{s Prelim}}

\subsection{The phase space and auxiliary spaces
\label{ss sp}}

As usual, we denote by $\Z$, $\Z^+$, and $\Z^-$ the set of all integers,
the set of all nonnegative integers, and the set of all nonpositive integers,
respectively.
We shall sometimes write $\Z^+_\tau$ to denote the infinite interval of
integer numbers in $[\tau, +\infty)$, so $\N = \Z_1^+$.
We use the convention that the sum equals zero if the lower index exceeds
the upper index
\begin{equation} \label{e sumconv}
\sum_k^j = 0 \quad \text{for} \quad j < k.
\end{equation}

For a seminormed space $\V$ with
a seminorm $|\cdot |_{\V}$, let $\Ssp ( \V )$ ($\Ssp_\pm ( \V )$)
denote the vector space of all
functions $v:\Z \to \V$ (resp., $v:\Z^\pm \to \V$).
We will also use the following standard spaces:
\begin{eqnarray*}
\l^p ( \V ):= \l^p (\Z^+, \V ) =  \left\{ v \! : \! \Z^+ \to \V  \ : \
\| v \|_p^p := \sum_{n=0}^{\infty} | v(n) |_{\V}^p  < \infty \ \right\}, \quad 1\leq p<\infty, \\
\l^{\infty} (\V ):= \l^{\infty} (\Z^+, \V ) =  \left\{ v \! :
\! \Z^+ \to \V \ : \ \| v \|_{\infty} := \sup_{n \in \Z^+} | v(n) |_\V <\infty \ \right\}.
\end{eqnarray*}
Recall that the $\l^p$-spaces are Banach spaces if $\V$ is a Banach space, and that they are
connected
by the continuous embedding
\begin{equation} \label{e ||q<||p}
\l^p ( \V ) \subseteq \l^q ( \V ) , \quad
%\text{and} \quad
\| v \|_q \leq \| v \|_p \qquad \text{if} \quad 1 \leq p \leq q \leq \infty .
\end{equation}

Let a (real or complex) Banach space $\X$ with a norm $| \cdot |$ be our basic space.
For a definition of the concept of a phase space we use
the vector space $\X^{\Z^-}$ of semi-infinite
tuples with elements in $\X$ and indices in $\Z^-$.
%In the context of this paper,
It is convenient to understand vectors of $\X^{\Z^-}$ as vector-columns.
That is, $\vphi \in \X^{\Z^-}$
has the form
\[
\vphi = \col \left( \vphi^{[m]} \right)_{m=-\infty}^{0} = \left( \begin{array}{r}
\vphi^{[0]} \\ \vphi^{[-1]} \\
 %\vphi^{-2} \\
 \dots \\  \vphi^{[m]} \\ \dots \end{array} \right) , \quad \text{where} \quad \vphi^{[m]} \in \X , \
m \in \Z^- \ .
\]
We will say that $\vphi^{[m]}$ is \emph{the m-th coordinate} of $\vphi$. (The notation of \cite{M97},
where function space $\Ssp_- (\X)$ is used instead of $\X^{\Z^-}$ for prehistory vectors,
can be considered as standard, but it is inconvenient for purposes of Section \ref{s RedOrd} of
the present paper.)

Our main objects are the system (\ref{e nh}) of nonhomogeneous linear functional difference
equations and the associated homogeneous system
\begin{equation}
x(n+1) = L (n) x_n ,  \quad n \in \Z^+. \label{e h}
\end{equation}
In formula (\ref{e nh}), $x(\cdot)$, $L (\cdot)$, $x_\bullet$, and $f (\cdot)$
have the following meaning:
\begin{description}
\item[\quad $\diamond$] $ \qquad x = x (\cdot) : \Z \to \X, \qquad f = f(\cdot): \Z^+ \to \X \ , 
\qquad x_\bullet : \Z \to \X^{\Z^-} $.
\item[\quad $\diamond$] The value $x_n \in \X^{\Z^-}$ of the function $x_\bullet$ at $n$ is
the prehistory of $x(n)$, i.e.,
\begin{eqnarray}
x_n^{[m]} & := & x (n+m) , \qquad m \in \Z^- ,   \notag \\
\text{and} \qquad x_n & = & \col ( x_n^{[m]} )_{m=-\infty}^{0} = \left( \begin{array}{r}
x_n^{[0]} \\ x_n^{[-1]} \\
%x_n^{-2} \\
\dots \\  x_n^{[m]} \\ \dots \end{array} \right) := \left( \begin{array}{r}
x (n) \\ x (n-1) \\
%x (n-2) \\
\dots \\  x (n+m) \\ \dots \end{array} \right) . \label{e xnm}
\end{eqnarray}
\item[\quad $\diamond$] For each $n \in \Z^+$, $L(n) : \dom L(n) \to \X$ is a linear map defined on a
certain linear subspace $\dom L(n) $ of $\X^{\Z^-}$. It is assumed also that the operator valued function
$L=L (\cdot)$
    \emph{defines} system  (\ref{e nh}) on a certain \emph{phase space} in the sense explained below.
\end{description}

\begin{definition} \label{d phsp}
A linear subspace
$\B \subseteq \X^{\Z^-}$ is called \emph{a phase space}
if for any $x:\Z \to \X$, the inclusion $x_0 \in \B$ implies
$x_n \in \B$ for all $n \in \Z^+$.
\end{definition}

An example of a phase space is the vector space
$\B_{\fin}$
%:=\{ \vphi \in \X^{\Z^-} : \vphi^{[m]}= \}$
of all \emph{finite vectors} of $\X^{\Z^-}$, that is, of all
$\vphi \in \X^{\Z^-}$ such that $\vphi^{[-j]} $ is zero for $j$ large enough. It is easy to see
that $\B_\fin \subseteq \B$ for any phase space $\B$.
A phase space $\B$ is usually assumed to be equipped with a semi-norm or a norm $| \cdot |_{\B}$ that
satisfies certain axioms (see e.g. \cite{MM05,M97}).
%In this paper we consider only two specific types
%of phase spaces with norm structures
%described at the end of this subsection.

Let $\V_1$ be a seminormed subspace of a certain vector space $\V$ and let $\V_2$ be a normed space.
Let $J$ be a linear mapping from a linear manifold $\dom J \subseteq \V$ to $\V_2$.
We write $J \in \L (\V_1 , \V_2 )$ and say that $J$ is a bounded linear operator from
$\V_1$ to $\V_2$ if $\V_1 \subseteq \dom J $ and
$\| J \|_{\V_1 \to \V_2} := \sup \{\, |J v|_{\V_2} \, : \,
|v|_{\V_1} \leq 1 \}$ is finite. If $\V_1 = \V_2$ is a normed space,
we write $\L(\V_2) :=\L (\V_1,\V_2)$.

\begin{definition} \label{d Ldef}
We shall say that the operator valued function $L (\cdot)$ \emph{defines system} (\ref{e nh}) on
a phase space $\B$
if $ \B \subseteq \dom L (n) $ for all $n \in \Z^+$.
If the phase space $\B$ is a seminormed space, we assume additionally that $L (n) \in \L (\B,\X)$ for
all $n \in \Z^+$.
\end{definition}

We will use normed and seminormed phase spaces of the following types:
Banach spaces $\B^{\ga}$ defined for $\ga \in \R$ by
\[
\B^{\ga} := \{ \vphi = \col ( \vphi^{[m]} )_{m=-\infty}^{0} \in \X^{\Z^-} \ :
\ | \vphi |_{\B^{\ga}} := \sup_{m \in \Z^-} |\vphi^{[m]} | e^{\ga m} < \infty \} ,
\]
and the seminormed linear spaces $\B_{[j,0]}^{\ga}$
defined for $j \in \Z^-$ and $\ga \in \R$ by
\[
\B_{[j,0]}^{\ga} = \X^{\Z^-} , \qquad | \vphi |_{\B_{[j,0]}^{\ga}} :=
\sup_{j \leq m \leq 0} |\vphi^{[m]} | e^{\ga m}  .
\]
It is clear that $\B_{[j,0]}^{\ga}$ is complete, that is, the quotient
$\B_{[j,0]}^{\ga} / | \cdot |_{\B_{[j,0]}^{\ga}} $ is a Banach space.

For phase spaces of $\l^p$-type see Section \ref{s Disc}.

\subsection{Solutions of initial value problems, stability
and first order systems} \label{ss sol}

Let $\W$ be an auxiliary Banach space.
The zero vectors of spaces $\X$ and $\W$ are denoted by $\vz_\X$ and $\vz_\W$, respectively.
The zero vector of
the vector space $\X^{  \Z^-}$ is denoted by $\vz_\B$.
% (since it is the zero vector of any phase space $\B$).
The zero function of the function space $\Ssp_+ (\V)$ (and its subspaces $\l^p (\V)$) will be denoted by
$\fz$ for any choice of $\V$.

From now on  we assume that the function $L$ defines system (\ref{e nh}) on a phase space $\B$.
For any $(\tau,\vphi) \in \Z^+ \times \B$, there exists  unique $x : \Z \to \X$
such that $x_\tau = \vphi$ and the relation (\ref{e nh}) holds for all $n \geq \tau$.
The function $x$ is called a solution of (\ref{e nh}) through
$(\tau,\vphi)$, and is denoted by $x (\cdot, \tau, \vphi; f)$.
For each $n \in \Z$, $x_n (\tau, \vphi; f) $ is the prehistory vector-column generated
by $x (j, \tau, \vphi; f)$, $-\infty <j \leq n$ in the way shown by (\ref{e xnm}).

\begin{definition} \label{d lplq st}
 The nonhomogeneous system (\ref{e nh}) is called \emph{$\l^p$-input $\l^q$-state
stable ($(\l^p,\l^q)$-stable, in short)} if $x(\cdot , 0, \vz_\B; f) \in \l^q (\X)$ for
any $f \in \l^p (\X)$.
\end{definition}

The following definition is a modification of standard ones, see e.g.
\cite[Section 7.2]{HMN91} and \cite{BB_JMAA05,BB_UnbDel}.
\begin{definition} \label{d UESn}
Assume that the function $L$ defines system (\ref{e nh}) on a seminormed phase space $\B$.

\textbf{(1)} The system (\ref{e h}) is called \emph{uniformly exponentially stable
(UES) in (the sense of) $\X$ with respect to the phase space $\B$}  if
there exist $K \geq 1$ and $\nu >0$ such that
\begin{equation}
|x(n, \tau, \vphi; \fz)| \leq  K e^{-\nu (n-\tau)} | \vphi |_{\B}
\quad \text{for all} \ n, \tau \ \text{such that}
\ n \geq \tau \geq 0.
\label{e ExStX}
\end{equation}
(If the phase space $\B$ is fixed we will say
in brief that the system is UES in $\X$.)

\textbf{(2)} The system (\ref{e h}) is called \emph{UES in
(the sense of) $\B$} if the $\X$-norm $|x(n, \tau, \vphi; \fz )|$
in (\ref{e ExStX}) is replaced by the
$\B$-seminorm $| x_n (\tau, \vphi; \fz ) |_{\B}$.
\end{definition}

We will use some stability results for a
\emph{first order difference system}
\begin{equation}
z(n+1) = \A (n) z (n) ,  \quad n \in \Z^+,
%\quad \text{where} \ z : \Z^+ \to \W, \ \A : \Z^+ \to \L (\W),
\label{e ordh}
\end{equation}
where $ z : \Z^+ \to \W$, $\A : \Z^+ \to \L (\W)$, and
$\W$ is a certain Banach space.
For $\psi \in \W$ and a function $\s : \Z^+ \to \W$, we denote by
$z (\cdot, \tau, \psi; \s): \Z_\tau^+ \to \W$
the solution of the associated nonhomogeneous initial value problem
\begin{eqnarray}
z(n+1) &= &\A (n) z (n) + \s (n),  \quad n \geq \tau \geq 0, %\qquad \s : \Z^+ \to \W,
\label{e ordnh} \\
z(\tau) & = &\psi \ , \qquad \psi \in \W .
\label{e ztau}
\end{eqnarray}

\begin{definition} \label{d ordUES}
\textbf{(1)} The homogeneous system (\ref{e ordh}) is called \emph{UES}  if there exist
$K_1 \geq 1$ and $\nu_1 >0$ such that
the solution of (\ref{e ordh}), (\ref{e ztau}) satisfies
\begin{equation}
|z(n, \tau, \psi; \fz)|_\W \leq  K_1 e^{-\nu_1 (n-\tau)}
| \psi |_\W  \quad \text{for all} \ n \geq \tau \geq 0.
\label{e ExStW}
\end{equation}

\textbf{(2)} The nonhomogeneous system (\ref{e ordnh}) is called \emph{$(\l^p,\l^q)$-stable}
if $z(\cdot , 0, \vz_\W; \s) \in \l^q (\W)$ for any $\s \in \l^p (\W)$.
%Here as before
%$ z(\cdot , 0, \vz_\W; \s)$ is the solution of
%(\ref{e ordnh}), (\ref{e ztau}) with $(\tau,\psi) = %(0,\vz_\W)$.
\end{definition}

For first order systems the following criterion is known.

\begin{theorem}[\cite{SS04}, see also \cite{P87,P88,AVM96} for particular cases] \label{t ordUES}
 Let $\A : \Z^+ \to \L (\W)$, let $1 \leq p \leq q \leq \infty$, and let the pair $(p,q)$
 be distinct from $(1,\infty)$.
  Then the homogeneous system (\ref{e ordh}) is UES if and only if the associated
  nonhomogeneous system (\ref{e ordnh}) is $(\l^p,\l^q)$-stable.
 \end{theorem}

%\begin{remark}
To the best of our knowledge, the general case of Theorem \ref{t ordUES} was first proved in
\cite{SS04}.
For $p=q=\infty$, the theorem was obtained earlier in \cite[Section 4]{P88}.
But it is easy to see that for the more general case $1 < p \leq q = \infty$
the essential part of the theorem (the implication 'if') follows from
\cite[Section 4]{P87} (one can check that (\ref{e ordnh})
is uniformly equicontrollable in the terms of \cite[Section 2]{P87}).
Note also that the remark at the end of \cite{P87} states  that the result of
\cite[Section 4]{P87} is contained implicitly in \cite{CS67}.

By a different method based on \cite{AVMZ94} the case $ 1 \leq p=q \leq \infty$
was proved in \cite[Corollary 5]{AVM96}, formally under the additional assumption
$\sup_{n \in \Z^+} \| \A (n) \|_{\W \to \W} < \infty$ (see also \cite{NN05,P04}).
This additional assumption can be easily
removed (see Remark \ref{r Alinf}).
Note that the papers \cite{AVM96,AVMZ94,NN05} study stability along with with exponential dichotomy.

The following statement is standard, cf. Proof of Theorem 2.2 in \cite{SS04}
and \cite[Lemma 2.3]{NN05} for the case of  first order systems.

\begin{proposition} \label{p GaL}
Assume that $1 \leq p,q \leq \infty$, $\ga \in \R$, and function
$L:\Z^+ \to \L (\B^\ga,\X) $ defines system (\ref{e nh}).
If (\ref{e nh}) is $(\l^p,\l^q)$-stable, then
\begin{equation} \label{e |x|q<|s|p}
\| x(\cdot , 0, \vz_\B; f) \|_q \leq K_{p,q,L}  \| f \|_p
\end{equation}
for a certain constant  $K_{p,q,L} \geq 1$
depending on $L$.
\end{proposition}

\begin{proof}
The linear operator $\Gamma : f \to x(\cdot , 0, \vz_\B; f)$
is correctly defined as an operator from $\l^p (\X)$ to $\l^q (\X)$.
It follows easily from (\ref{e nh}) that
$\Gamma$ is closed. By the closed graph principle, $\Gamma$ is bounded.
\end{proof}

%Note that the constant $K_{p,q,L}$ cannot be chosen less than $1$.
%For this purpose it is enough to take $f(0) \neq \vz_\X$ and $f(n) = \vz_\X$ for all
%$n \geq 1$.

We need also an analogue of the above proposition
for the first order system (\ref{e ordnh}). Namely, if $1 \leq p , q \leq \infty$ and
$\A : \Z^+ \to \L (\W)$, then
\begin{equation} \label{e |z|q<|s|p}
(\l^p,\l^q)\text{-stability of (\ref{e ordnh}) implies} \quad
%\quad \Longrightarrow \quad
\| z(\cdot , 0, \vz_\W; \s) \|_q \leq K_{p,q,\A} \| \s \|_p \, \end{equation}
with a certain constant $K_{p,q,\A} \geq 1$.
The proof is the same.

\begin{remark} \label{r Alinf}
\textbf{(1)} UE stability of (\ref{e ordh}) immediately implies
%\begin{equation} \label{e Alinf}
$ \sup_{n \in \Z^+} \| \A (n) \|_{\W \to \W} < \infty$.
%\end{equation}
Indeed, from (\ref{e ExStW}), we have
\[
| \A(\tau) \psi |_\W = | z(\tau+1,\tau,\psi; \fz) |_\W
 \leq K_1 e^{-\nu_1} \| \psi \|_\W ,
\quad \psi \in \W, \quad \tau \in \Z^+ .
\]
So $\| \A (n) \|_{\W \to \W} \leq K_1 e^{-\nu_1}$ for all $n \in \Z^+$.

\textbf{(2)} On the other hand, $(\l^p,\l^q)$-stability of (\ref{e ordnh}) also yields
$ \sup_{n \in \Z^+} \| \A (n) \|_{\W \to \W} < \infty$.
In fact, for $\psi \in \W$, consider $\s \in \l^p (\W)$ defined by
$\s (k) = \delta_{k,n} \psi$, where $\delta_{k,n}$ is Kronecker's delta.
Then $z (n+1, 0, \vz_\W; \s) = \psi$, $\s (n+1) = \vz_\W$, and therefore
$z (n+2, 0, \vz_\W; \s ) = \A (n+1) \psi $.
By (\ref{e |z|q<|s|p}),
\[
| \A (n+1) \psi |_\W = | z (n+2, 0, \vz_\W; \s) |_\W \leq
\| z (\cdot, 0, \vz_\W ; s) \|_q \leq
K_{p,q,\A} \| s \|_p = K_{p,q,\A} | \psi |_\W \ .
\]
\end{remark}

\subsection{Auxiliary operators and related notation\label{ss op}}

Let $\B$ be a phase space. Let $\I$ stand for the identity operator in $\X^{\Z^-}$ and
so for the identity operator in $\B$.

For functions $L:\Z^+ \to \L (\B,\X)$ and $g : \Z^+ \to \B$, we define the function $L g$ by
\begin{equation*}
L g : \Z^+ \to \X, \qquad (Lg) (n) := L(n) g (n), \quad n \in \Z^+ .
\end{equation*}
If $C$ is a map from $\B$ to $\X$, then the function $C g : \Z^+ \to \X$ has the natural meaning of
\[
(C g) (n)= C g(n), \quad n \in \Z^+ .
\]

For $m \in \Z^-$ and $g : \Z^+ \to \B^\ga$, we will use the shortening $g^{[m]}$ for
the function that maps $n \in \Z^+$ to the $m$-th coordinate $(g(n))^{[m]}$ of $g(n)$, i.e.,
\begin{equation} \label{e gm}
g^{[m]} : \Z^+ \to \X, \qquad g^{[m]} (n) := (g (n) )^{[m]} .
\end{equation}

Let us define the 'backward shift' operator
$\Sh : \X^{\Z^-} \to \X^{\Z^-}$  by
\begin{eqnarray*} \label{e Sh}
(\Sh \vphi)^{[m]}  :=  \left\{
\begin{array}{rr}
\vz_{\X} , &  m = 0 \\
\vphi^{[m+1]}  , & m \leq -1
\end{array} \right. \ , \quad \vphi \in \X^{\Z^-} .
%\qquad \text{for} \quad j \geq 0 \ .
\end{eqnarray*}
The operators $\Sh^j $, $j \in \N$, shift coordinates of $\vphi$ on $j$
units downward and supplement coordinates with indices from $1-j$
to $0$ by the zero-element $\vz_\X$. As usual, $\Sh^0 = \I$.

For $m_1 \in \{-\infty\} \cup \Z^- $ and $m_2 \in \Z^-$ such that $m_1 \leq m_2 $,
we define the projection operator
$\Pr_{[m_1,m_2]} : \X^{\Z^-}  \to \X^{\Z^-}$ by
\begin{equation} \label{e Prm1m2}
(\Pr_{[m_1,m_2]} \vphi)^{[m]} = \left\{ \begin{array}{rr}
\vphi^{[m]} , & m_1 \leq m \leq m_2 \\
\vz_{\X}  , & \text{otherwise}
\end{array} \right. \ , \qquad  m \in \Z^- , \quad  \vphi \in \X^{\Z^-}.
\end{equation}
The operator $\Pr_{[m_1,m_2]}$ saves the coordinates from $m_1$ to $m_2$
and nulls all other coordinates.
If $m_1=m_2 = m\in \Z^-$, we write
\[
\Pr_{\{m\}} := \Pr_{[m,m]}.
\]
We will use extensively the operator $\Pr_{\{0\}} $ that maps
$\col (\vphi^{[0]}, \vphi^{[-1]}, \vphi^{[-2]}, \dots)$ to
$\col (\vphi^{[0]}, \vz_\X , \vz_\X, \dots)$.

Note that for any $\vphi \in {\B^{\ga}}$,
\begin{eqnarray}
| \Sh^j \vphi |_{\B^{\ga}} & = & e^{-j\ga} | \vphi |_{\B^{\ga}} \ , \quad j \in \Z^+, \label{e |Sh|} \\
| \vphi |_{\B^{\ga}} & = &
\max \{ \ |\Pr_{\{0\}} \vphi |_{\B^{\ga}} \ , \ | (\I - \Pr_{\{0\}}) \vphi |_{\B^{\ga}} \ \} .
\label{e |phi|P0}
\end{eqnarray}

For $j \in \Z^-$ we consider also the operators
\begin{equation} \label{e Ej}
 \E_j : \X \to \X^{\Z^-} , \qquad
 (\E_j \psi)^{[m]}  = \left\{ \begin{array}{rr}
\psi , & m = j \\
\vz_{\X}  , & m \neq j
\end{array} \right. , \qquad \psi \in \X, \ \ m \in \Z^-.
\end{equation}
Assume that $\dom L(n) \supseteq \B_{\fin}$ for all $n$.
Let us define components of the operator $L(n)$ by
\begin{equation} \label{e Lnm}
L(n,k) : \X \to \X, \qquad L(n,k) := L (n) \E_{-k} , \quad \ k \in \Z^+ \ .
\end{equation}

\begin{remark} \label{r Lnk not def}
Note that generally for $L(n) \in \L(\B^\ga,\X)$, operators $L(n,k)$, $k \in \Z^+$,
do not determine the operator $L(n)$ on $\B^\ga$.
As an example, one can take $\X=\R$, $\ga=0$, and
$L(n)$ equal to any of Banach limits (see e.g. \cite[Sec. II.4.22]{DSh58} for the definition).
Then $L(n,k) = 0$ for all $k$, but $L(n) \col (1,1, \dots) =1$.
\end{remark}

However, $L(n,k)$ determine $L(n)$ on finite vector-columns in the following way: for any
$\vphi \in \B_\fin$,
\[
L (n) \vphi = \Bigl( \, L (n,0) \ L (n,1) \ \dots \ L (n,-m) \ \dots  \ \Bigr) \left( \begin{array}{r}
\vphi^{[0]} \\ \vphi^{[-1]} \\
 %\vphi^{-2} \\
 \dots \\  \vphi^{[m]} \\ \dots \end{array} \right)=
\sum_{m  \ : \ \vphi^{[m]} \neq \vz_\X } L (n,-m) \vphi^{[m]} \ .
\]

\section{Exponential stability and $(\l^p,\l^q)$-stability \label{s UES}}

The main result on UE stability, Theorem \ref{t UES}, and some of its
corollaries are presented in Section \ref{ss UESres}. The proof of Theorem
\ref{t UES} given in Section \ref{ss proofUES}
is based on the method described in Section \ref{s RedOrd}.

%\subsection{Exponential stability criterions and $(\l^p,\l^q)$-stability}
%\subsubsection{Systems with infinite delay.}

\subsection{Main results \label{ss UESres}}

Recall that the projection operators $\Pr_{[m_1,m_2]}$ were
defined in Section \ref{ss op}.

\begin{theorem} \label{t UES}
Let $\ga>0$ and let $L:\Z^+ \to \L (\B^\ga,\X)$ define system (\ref{e nh}). % on the phase space $\B^\ga$.
Assume that the pair $(p,q)$ is such that
\begin{equation}
1 \leq p \leq q \leq \infty \qquad \text{and} \qquad  (p,q) \neq (1,\infty) .
\label{e pqneq}
\end{equation}
Then the following statements are equivalent:
\begin{description}
\item[(i)] System (\ref{e h}) is UES in $\X$ with respect to (w.r.t.) $\B^\ga$.
\item[(ii)] System (\ref{e h}) is UES in $\B^{\ga}$.
\item[(iii)] System (\ref{e nh}) is $(\l^p,\l^q)$-stable and
\begin{equation} \label{e LPlinf}
\text{there exists } \ m \in \Z^- \ \ \text{such that} \quad
\| L (\cdot) \Pr_{[-\infty,m]} \|_\infty := 
\sup_{n \in \Z^+} \| L (n) \Pr_{[-\infty,m]} \|_{\B^\ga \to \X} < \infty \ .
\end{equation}
\end{description}
\end{theorem}

%\begin{remark} \label{r BB_UnbDel}
In the case $p=q=\infty$, Theorem \ref{t UES} was obtained in \cite{BB_UnbDel}
under certain additional conditions.
The method of \cite{BB_UnbDel} differs from the method of the present paper.
%\end{remark}

\begin{remark} \label{r |L|<inf}
The proof of Theorem \ref{t UES} shows that if any of statements (i)-(iii) of
Theorem \ref{t UES} is fulfilled, then
%\begin{equation} \label{e Llinf}
$\sup_{n \in \Z^+} \| L (n) \|_{\B^\ga \to \X} < \infty. $
%\end{equation}
\end{remark}

\begin{remark} \label{r exPmn}
\textbf{(1)}
%Theorem \ref{t UES} is a proper extension of Theorem \ref{t ordUES} on systems with infinite delay.
Simple Example \ref{ex lplq ind}
demonstrates that condition (\ref{e LPlinf})
in Theorem \ref{t UES} \emph{cannot be omitted}.
More subtle Example \ref{ex Pmn} shows that
condition (\ref{e LPlinf}) \emph{cannot be replaced} by the
less restrictive condition
\begin{equation} \label{e LPmn}
 \sup_{n \in \Z^+} \| L (n) \Pr_{[-\infty,m_n]} \|_{\B^\ga \to \X} < \infty \
\end{equation}
with non-positive $m_n$ such that $\lim\limits_{n\to\infty} m_n = -\infty$.

\textbf{(2)} Consider the case when $\ga\leq 0$ and $\X$ is nontrivial (i.e., $\X \neq \{ 0_\X\}$).
Then UE stability in $\B^\ga$ \emph{does not hold} for any system of the form (\ref{e nh}).
This follows immediately from the definitions of $\B^\ga$ and UE stability.
Example \ref{ex BB05 ex2} shows that, in general, the implication (iii)$\Rightarrow$(i) is also not valid.
\end{remark}

%\subsubsection{Independence of $(\l^p,\l^q)$ stabilities on the parameters $p$ and $q$.}

Since UE-stability does not depend on the choice of $p$ and $q$ in the $(\l^p,\l^q)$-stability
property we get the following.

\begin{corollary} \label{c lplq}
Let $\ga >0$ and let a function $ L:\Z^+ \to \L (\B^\ga,X)$ define system (\ref{e nh}).
Assume that (\ref{e LPlinf}) holds. Then $(\l^p,\l^q)$-stability of (\ref{e nh}) for a certain pair
$(p,q)$ satisfying (\ref{e pqneq}) implies $(\l^{p},\l^{q})$-stability of (\ref{e nh}) for all $(p,q)$
satisfying (\ref{e pqneq}).
\end{corollary}

\begin{remark} \label{r lplq}
Example \ref{ex lplq ind} shows that assumption (\ref{e LPlinf}) in Corollary \ref{c lplq}
\emph{cannot be dropped},
however (\ref{e LPlinf}) can be relaxed to the condition that
\begin{equation} \label{e LPlinf subd}
\text{there exists } \ m \in \Z^- \ \ \text{such that} \quad
 \sup_{n \geq -m+1} \| L (n) \Pr_{[-n+1,m]} \|_{\B^\ga \to \X} < \infty \ .
\end{equation}
Indeed, %since
$(\l^p,\l^q)$-stability of (\ref{e nh}) \emph{does not depend} on the parts
$L (n) \Pr_{[-\infty,-n]}$ of the operators $L(n)$ (see also Section \ref{s Ind ga}).
%, Corollary \ref{c lplq} is valid with assumption
%(\ref{e LPlinf subd}) instead of (\ref{e LPlinf}) .
\end{remark}

%\subsubsection{Systems with bounded delay.}

\begin{definition} \label{d bound del}
Assume that $L(n)$ is defined on $\B_{\fin}$ for all $n \in \Z^+$ and that
$L (n,k) \in \L (\X,\X)$ for all $n,k \in \Z^+$.
Then (\ref{e nh}) is called \emph{a system of  difference
equations with bounded delay} if there exists $m \in \Z^-$
such that
 \begin{equation} \label{e bound del}
 L (n) \Pr_{[-\infty,m]}  = 0 \qquad \text{for all} \quad n \in \Z^+.
 \end{equation}
 If $m$ is the largest (nonpositive) number such that (\ref{e bound del}) holds, then $|m|$ is called
 the order of the system (\ref{e nh}).
\end{definition}

If (\ref{e nh}) is a system
with bounded delay of order $d$, then it
can be written in the form
\begin{eqnarray*}
x(n+1) = \sum_{k=0}^{d-1} L (n,k) x (n-k) + f (n), \quad n \in \Z_+ ,
\end{eqnarray*}
 and it can be considered on the whole vector space $\X^{\Z^-}$.
%Indeed, each of operators $L(n)$ can be naturally extended to $\vphi \in \X^{\Z^-} \setminus \dom L(n)$
%by the equality $L(n) \vphi = L(n) \Pr_{[-d+1,0]} \vphi$. This extension, for which we use the same
%notation $L(n)$, obviously keeps fulfilling (\ref{e bound del}) and for any $\ga \in \R$ belongs to
%the class $\L (\B_{[-d+1,0]}^\ga, X)$ (defined in Section \ref{ss sp}). Since all
%seminorms of the spaces $\B_{[-d+1,0]}^\ga$, $\ga \in \R$, are equivalent, $\B_{[-d+1,0]}^0$ is
%the most natural phase space for a system with bounded delay of order $d$. However
Any of the spaces $\B^\ga$ or $\B_{[j,0]}^0$, with $j \leq -d+1$ and $\ga \in \R$,
can be chosen as a phase space.
 Note that UE stability in $\X$ does not depend on that choice due to the obvious equality
\begin{eqnarray}
\label{e x=x(P)}
x (\cdot, 0, \vphi ; f ) = x (\cdot, 0, \Pr_{[-d+1,0]} \vphi ; f ) , \quad \vphi \in \X^{\Z^-} .
\end{eqnarray}
Thus, Theorem \ref{t UES} implies the following result.

\begin{corollary} \label{c UESbd}
 Let $\ga \in \R$, $1 \leq p \leq q \leq \infty$, and the pair $(p,q)$ be distinct from $(1,\infty)$.
If (\ref{e nh}) is a  system with bounded delay of order $d$, then the following statements are equivalent:
\begin{description}
\item[(i)] System (\ref{e h}) is UES in $\X$ with respect to $\B^\ga$
(or, equivalently, w.r.t. $\B_{[-d+1,0]}^\ga$).
\item[(ii)] System (\ref{e nh}) is  $(\l^p,\l^q)$-stable.
\end{description}
\end{corollary}

\begin{proof}
For $\B^\ga$ with $\ga >0$, the corollary follows from Theorem \ref{t UES}
and the fact that (\ref{e LPlinf}) is fulfilled for any system with bounded delay.
Now note that the $\B^\ga$-norms are equivalent for all $\ga \in \R$ on the subspace of
 $\vphi \in \X^{  \Z^-}$ such that
$\vphi^m =\vz_\X$ for $m \leq  - d$ (that is, on the range of $\Pr_{[-d+1,0]}$).
Combining this and (\ref{e x=x(P)}) completes the proof.
\end{proof}

The connection between UE stability and $(\l^p,\l^q)$-stability of systems with bounded delay
was considered
in \cite{BB_FDE04,BB_JMAA05,BB_UnbDel}. Corollary \ref{c UESbd}
remove the assumption $\sup_{n \in \Z^+} \| L (n) \|_{\B^0 \to \X} < \infty$ imposed
in \cite{BB_FDE04,BB_JMAA05,BB_UnbDel} and extends the results of these papers to the case 
$1\leq p< q <\infty$.

\subsection{Reduction of order, representation theorem, and
auxiliary results \label{s RedOrd}}

Let $\ga \in \R$ and let $L:\Z^+ \to \L (\B^\ga,\X)$ define
system (\ref{e nh}) on $\B^\ga$.
In this subsection, we show that system (\ref{e nh})
can be written as a system of first order difference equations in the
space $\B^\ga$.

Recall that the operators $\Sh$, $\E_0$, and $\Pr_{\{0\}}$ are defined in Section \ref{ss op}.

Let us define operators $D (n): \B^{\ga} \to \B^{\ga}$ by
\begin{equation} \label{e D}
D (n) := \E_0 L(n) + \Sh , \quad n \in \Z^+ .
\end{equation}
%\[
%(D (n) \vphi)^k := \left\{
%\begin{array}{rr}
%L (n) \vphi , & k = 0 \\
%\vphi (k+1)  , & k \leq -1
%\end{array}
%\right. ,  \quad n \in \Z^+ .
%\]
It follows from $L(n) \in \L (\B^\ga,\X)$ and (\ref{e |Sh|}) that operators $D (n) $ are bounded in
$\B^{\ga}$ and
\begin{eqnarray} \label{e |D|}
\| D(n) \|_{\B^{\ga} \to \B^{\ga}} = \max \{ \ \| L (n) \|_{ \B^{\ga} \to \X } \ , \ e^{-\ga} \ \} .
\end{eqnarray}

Let $y$ and $g$ be functions from $\Z^+$ to $\B^{\ga}$.
Let $\vphi \in \B^{\ga}$ and $\tau \in \Z^+$.
Consider the initial value problem
\begin{eqnarray} \label{e ndnh}
y (n+1) & = & D (n) y(n) + g (n), \quad n \in \Z_\tau^+ \ ,
 \\
y(\tau) & = & \vphi \ , \label{e y0=phi}
\end{eqnarray}
and denote its solution by $y(\cdot , \tau, \vphi; g) $.
Then $y(n) = x_n (\tau,\vphi;f)$ is the solution of (\ref{e ndnh}), (\ref{e y0=phi}) with $g= \E_0 f$ .
In other words,
\begin{equation} \label{e xn=yn}
x_n (\tau,\vphi;f)=y(n,\tau,\vphi;\E_0 f) , \quad n \in \Z_\tau^+ .
\end{equation}

Let us show that it is possible to express
$y(\cdot , 0, \vz_\B; g) $ with general
$g : \Z^+ \to \B^{\ga}$ in terms of solutions of (\ref{e nh}) and shift operators.
Recall that we use the sum convention (\ref{e sumconv}) and that for $g : \Z^+ \to \B^\ga$,
the function that maps $n \in \Z^+$ to
the $m$-th coordinate $(g(n))^{[m]}$ of $g(n)$ is denoted by $g^{[m]}$.

\begin{proposition}\label{p rep}
For any $g : \Z^+ \to \B^{\ga}$,
\begin{eqnarray}\label{e rep1}
 y (n,0,\vz_\B;g) & = & x_n (0,\vz_\B;g^{[0]} + L \h) + \h (n),  \quad \text{where}  \\
 %\qquad \qquad
% \qquad
\h (n) = (H g) (n)& := & \sum_{k=0}^{n-1} \Sh^{n-k-1} (\I-\Pr_{\{0\}}) g (k) , \quad n \in \Z^+.
\label{e reph}
\end{eqnarray}
\end{proposition}

\begin{proof}
For $n = 0$, one can see that $\h(0) = x_0 (0,\vz_\B ; g^{[0]} + L \h) = \vz_\B$,
and therefore (\ref{e rep1}) is trivial.
For $n=1$, (\ref{e nh}) implies
\[
x (1,0,\vz_\B ; g^{[0]} + L \h) = L(0) \vz_\B + g^{[0]} (0) + L(0) \h (0) = g^{[0]} (0) \ .
\]
So the vector-column $x_1 (0,\vz_\B ; g^{[0]} + L \h)$ has $0$-th coordinate equal
to $g^{[0]} (0) $ and all other coordinates equal to zero,
i.e.,
\[
x_1 (0,\vz_\B ; g^{[0]} + L \h) = \E_0 g^{[0]} (0).
\]
Since $\h (1) := (\I-\Pr_{\{0\}}) g (0)$, we see that
\[
x_1  (0,\vz_\B ; g^{[0]} + L \h) + \h (1) =
\E_0 g^{[0]} (0) + (\I-\Pr_{\{0\}}) g (0) = g(0) = y (1,0,\vz_\B;g),
\]
and so (\ref{e rep1}) holds true for $n=1$.

Let us assume (\ref{e rep1}) for certain $n \in \N$ and prove it for $n+1$.
First, note that
\begin{eqnarray}
\h (n+1) = \Sh \h (n) + (I-\Pr_{\{0\}}) g (n) , \quad n \in \Z^+ . \label{e hn+1}
\end{eqnarray}
Indeed,
\begin{multline*}
 \Sh \h (n) + (I-\Pr_{\{0\}}) g (n)   =
 \Sh \sum_{k=0}^{n-1} \Sh^{n-k-1} (\I-\Pr_{\{0\}}) g (k) + (I-\Pr_{\{0\}}) g (n)
 = \\
 =
 \sum_{k=0}^{n-1} \Sh^{n-k} (\I-\Pr_{\{0\}}) g (k) + (I-\Pr_{\{0\}}) g (n)
  =
\sum_{k=0}^{n} \Sh^{n-k} (\I-\Pr_{\{0\}}) g (k) = \h (n+1) \ .
\end{multline*}
From (\ref{e xn=yn}) and (\ref{e ndnh}), one can get
\begin{eqnarray}
D (n) x_n (0,\vz_\B ; g^{[0]} + L \h) = x_{n+1} (0,\vz_\B ; g^{[0]} + L \h) -
 \E_0 \left[ g^{[0]} (n) + L(n) \h (n) \right] .\label{e xn+1}
\end{eqnarray}
Now we substitute (\ref{e rep1}) which is assumed to be valid for $n $ into
(\ref{e ndnh}) and get
\begin{eqnarray}
y (n+1,0,\vz_\B ; g) = D (n) x_n (0, \vz_\B ; g^{[0]} + L \h) + D (n) \h (n)  + g(n) \ . \label{e yn+1}
\end{eqnarray}
Modifying the last two terms with the use of (\ref{e D}),  we get
\begin{eqnarray*}
D (n) \h (n) + g (n) & = & \E_0 L(n) \h (n) + \Sh \h (n) + g(n) = \\ & = &  \E_0 L(n) \h (n) + \Sh \h (n)
 +   \E_0 g^{[0]} (n) + (I-\Pr_{\{0\}}) g (n) .
\end{eqnarray*}
Equality (\ref{e hn+1}) implies
\begin{eqnarray*}
D (n) \h (n) + g (n) & = & \bigl[ \E_0 L(n) \h (n) +  \E_0 g^{[0]} (n)\bigr] + \bigl[\Sh \h (n)
 + (I-\Pr_{\{0\}}) g (n) \bigr]
= \\ & = & \E_0 \left[ L(n) \h (n) +  g^{[0]} (n) \right] + \h (n+1) .
\end{eqnarray*}
Substituting the last equality and (\ref{e xn+1}) into
(\ref{e yn+1}), we get
\[
y (n+1,0,\vz_\B ; g) = x_{n+1} (0, \vz_B ;  g^{[0]} + L \h) + \h (n+1)  \ .
\]
This is equality (\ref{e rep1}) for $n+1$.
Induction completes the proof.
\end{proof}

\begin{proposition} \label{p hlp}
Let $\ga >0$, $1 \leq p \leq \infty$, $g \in \l^p (\B^\ga)$, and let $\h$ be the function defined in 
(\ref{e reph}).
Then \
%\[
$ \| \h \|_p \leq (1-e^{-\ga})^{-1} \| g \|_p $.
%\]

\end{proposition}

\begin{proof}
Recall that $h(0) = 0$ and consider $n \in \N$.
Then for $m \in \Z^-$, $m$-th coordinate of $\h(n)$ can be written
in the following way
\begin{eqnarray*}
\h^{[m]} (n) &  = \sum_{k=0}^{n-1} \bigl( \Sh^{n-k-1} (\I-\Pr_{\{0\}}) g (k) \bigr)^{[m]} =
\sum_{k=0}^{n-1} \left\{
\begin{array}{ll}
0 , & m \geq - n + k +1 \\
g^{[m+n-k-1]} (k)  , & m \leq - n + k
\end{array} \right.
\end{eqnarray*}
Since nonzero terms in the last sum correspond to
$k \in \Z^+$ such that $k \geq m+n$, we have
\begin{equation} \label{e hm=sumg}
\h^{[m]} (n) = \sum_{k=\max \{0, m+n\} }^{n-1}  g^{[m+n-k-1]} (k) \ .
\end{equation}
(Due to the sum convention (\ref{e sumconv}) this formula is also valid
for $n=0$).

Using the last formula, we estimate $| \h (n) |_{\B^\ga}$ and then $\| h \|_p$:
\begin{eqnarray}
| \h (n) |_{\B^\ga} & = & \sup_{m \in \Z^-} e^{m\ga}
\left| \sum_{k=\max \{0, m+n \} }^{n-1}  g^{[m+n-k-1]} (k) \right|
\leq \notag \\ & \leq &
\sup_{m \in \Z^-}
 \sum_{k=\max \{ 0, m+n \} }^{n-1} e^{(-n+k+1)\ga} e^{(m+n-k-1) \ga}  \left| g^{[m+n-k-1]} (k) \right|
\leq \notag \\ & \leq &
\sup_{m \in \Z^-}
 \sum_{k=\max \{ 0, m+n \} }^{n-1} e^{(-n+k+1)\ga}
 | g (k) |_{\B^\ga} =
 \sum_{k=0}^{n-1} e^{-((n-k)-1)\ga} | g (k) |_{\B^\ga} =
 (\e * \g) (n) , \label{e est h(n)}
\end{eqnarray}
where $ (\e * \g) ({\scriptstyle \bullet}) =
\sum\limits_{k=0}^{ \bullet } \e ({\scriptstyle \bullet} - k) \g (k)$
is the discrete convolution of the functions
\begin{equation}
\g \in \l^p (\R), \quad \g (n) := | g (n) |_{\B^\ga} , \quad
\text{and} \quad \e \in \l^1 (\R), \quad
\e (n) = \left\{
\begin{array}{rr}
0 , & n = 0 \\
e^{-(n-1)\ga}  , & n \geq 1
\end{array} \right. . \label{e def g e}
\end{equation}
Using Young's inequality for convolutions (see e.g. \cite[Problem VI.11.10]{DSh58}),
we get
\[
\| \h \|_p \leq \| \e * \g \|_p \leq \| \e \|_1 \| \g \|_p =
(1-e^{-\ga})^{-1} \| g \|_p \ .
\]
\end{proof}

\begin{proposition} \label{p lp}
Let $1 \leq p \leq \infty$, $\ga >0$, and let $x_0 = [x(m)]_{m=-\infty}^{0}$ belong to $ \B^\ga$. Then
\[
x_\bullet  \in \l^p (\Z^+,\B^{\ga}) \quad \text{if and only if} \quad x (\cdot) \in \l^p (\Z^+,\X).
\]
 More precisely,
\begin{eqnarray}
\| x (\cdot) \|_\infty & \leq  \| x_\bullet \|_\infty & =
\max \{ | x_{0} |_{\B^\ga} , \| x (\cdot) \|_\infty \},
\label{e ||<||linf}\\
\| x (\cdot) \|_p^p & \leq  \| x_\bullet \|_p^p & \leq
\frac {1} {1 - e^{-p\ga}}
\left( | x_{0} |_{\B^\ga}^p + \| x (\cdot) \|_p^p \right),
\quad 1 \leq p < \infty .
\label{e ||<||lp}
\end{eqnarray}
\end{proposition}

\begin{proof}
Formula (\ref{e ||<||linf}) and the first inequality in
(\ref{e ||<||lp}) are obvious.
Let us prove the second inequality in
(\ref{e ||<||lp}).
Note that
$| x_{-1} |_{\B^\ga} \leq e^\ga | x_0 |_{\B^\ga}$ and
for $n \in \Z^+$,
\[
| x_n |_{\B^\ga}^p = \sup_{m \in \Z^-}    |e^{\ga m} x(n+m)|^p
\leq e^{p \ga (-n-1)} | x_{-1} |_{\B^\ga}^p + \sum_{m=-n}^{0} e^{p \ga m} | x(n+m) |^p  .
\]
Therefore,
\begin{eqnarray*}
\sum_{n=0}^{+\infty} | x_n |_{\B^\ga}^p
 & \leq &
| x_{-1} |_{\B^\ga}^p \sum_{n=0}^{+\infty} e^{-p \ga (n+1)}
+ \sum_{m=-\infty}^0 e^{p \ga m} \sum_{n=-m}^{+\infty} |x(n+m)|^p
= \\ & = &
\frac {e^{-p\ga}} {1 - e^{-p\ga}} | x_{-1} |_{\B^\ga}^p
+ \frac {1}{1 - e^{-p\ga}} \| x (\cdot) \|_p^p \leq
\frac {1} {1 - e^{-p\ga}}
\left( | x_{0} |_{\B^\ga}^p + \| x (\cdot) \|_p^p \right) .
\end{eqnarray*}
\end{proof}

\begin{remark} \label{r lp}
Clearly, if $\ga =0$, the proposition is valid only for $p = \infty$.
In this case (\ref{e ||<||linf}) still holds.
\end{remark}

Note that UE  stability of (\ref{e h}) in $\B^{\ga}$ coincides with UE stability
(in the sense of Definition \ref{d ordUES}) of the homogeneous system
\begin{equation}
\label{e ndh}
y (n+1) = D (n) y(n) , %\quad n \in \Z^+ ,
 \qquad y \in \Ssp_+ (\B^{\ga}) \ .
\end{equation}
corresponding to (\ref{e ndnh}).
 On the other hand system (\ref{e h}) is UES in $\X$ if it is UES in $\B^{\ga}$.
If $\ga>0$ the converse is also true (cf. \cite[Section 7.2]{HMN91}
for the differential equation case).

\begin{proposition} \label{p UES}
Let $\ga>0$. Then the following statements are equivalent
\begin{description}
\item[(i)] System (\ref{e h}) is UES in $\X$ w.r.t. $\B^{\ga}$.
\item[(ii)] System (\ref{e h}) is UES in $\B^{\ga}$.
\item[(iii)] System (\ref{e ndh}) is UES.
\end{description}
\end{proposition}

\begin{proof}
We have only to prove that (\ref{e ExStX}) implies
(\ref{e ExStW}) with $\W=\B^\ga$ for the solution $y(n, \tau, \vphi; \fz) $
of (\ref{e ndh}).

Indeed,
\[
|x (n, \tau, \vphi; \fz)| = |\Pr_{\{0\}} y(n, \tau, \vphi; \fz)
|_{\B^{\ga}},
\]
so (\ref{e ExStX}) can be rewritten as
\begin{equation*}
|\Pr_{\{0\}} y(n,\tau, \vphi; \fz) |_{\B^{\ga}} \leq  K e^{-\nu (n-\tau)} | \vphi |_{\B^{\ga}}  ,
\quad n \geq \tau \geq 0,
%\label{}
\end{equation*}
(recall that $K \geq 1$ and $\nu >0$).
Note that
\begin{equation*}
| (\I - \Pr_{\{0\}}) y(n,\tau, \vphi; \fz) |_{\B^{\ga}}
 =  e^{-\ga} |  y(n-1,\tau, \vphi; \fz) |_{\B^{\ga}} .
\end{equation*}
Using (\ref{e |phi|P0}), we see that the assumption
\[
| y(n-1,\tau, \vphi; \fz) |_{\B^{\ga}} \leq  K e^{- \nu_1  (n-1-\tau) } | \vphi |_{\B^{\ga}}, \quad
\text{where} \ \nu_1 := \min \{\nu , \ga \},
\]
implies
\[
| y(n,\tau, \vphi; \fz) |_{\B^{\ga}}
\leq \max \{ \, K e^{-\nu (n-\tau)} | \vphi |_{\B^{\ga}} \ ,
\ K e^{-\ga} e^{- \nu_1  (n-1-\tau) } | \vphi |_{\B^{\ga}}  \, \} \leq
K e^{- \nu_1  (n-\tau) } | \vphi |_{\B^{\ga}} .
\]
Induction completes the proof.
\end{proof}

Recall that the operators $L (n,k)$ are defined in Section \ref{ss op}.

\begin{proposition} \label{p |Lnk|<}
Let $\ga \in \R$, $L:\Z^+ \to \L (\B^{\ga},\X)$,
%that (\ref{e LnLBX}) holds,
and $1 \leq p,q \leq \infty$. Assume that (\ref{e nh}) is $(\l^p,\l^q)$-stable. Then:
\begin{description}
\item[(i)] For any $k \geq 0$ and $n \geq k+1$,
\begin{eqnarray*}
\| L (n,k) \|_{\X \to \X} \ \leq \ 2^{k} \ K_{p,q,L}^{k+1},
\end{eqnarray*}
where $K_{p,q,L} $ is the constant introduced in Proposition \ref{p GaL}.
\item[(ii)] For any $j \in \Z^-$,
\begin{eqnarray} \label{e supLP<inf}
%\| L(\cdot) \Pr_{[j,0]} \|_\infty :=
\sup\limits_{n \in \Z^+} \| L(n) \Pr_{[j,0]} \|_{\B^\ga \to \X} < \infty \ .
\end{eqnarray}
 \end{description}
\end{proposition}

\begin{proof}
\textbf{(i)}
We use induction in $k$ to prove that
\begin{eqnarray} \label{e L(n+k,k)}
\| L (n_0+k, k) \|_{\X \to \X} \leq 2^{k} K_{p,q,L}^{k+1}.
\end{eqnarray}
for any $n_0 \in \N$ and $k \in \Z^+$.

Let $k=0$. Then (\ref{e L(n+k,k)}) can be obtained in the way shown in Remark \ref{r Alinf} (2).
Assume now that $k_1 \in \N$ and (\ref{e L(n+k,k)}) holds for all $0 \leq k \leq k_1 -1$.

For any $\psi \in \X$ and  $n_0 \in \N$ we can choose $f \in \l^p (\X)$ such that
\begin{eqnarray} \label{e xn=psi}
x_{n_0-1} (0,\vz_\B; f) = \vz_\B , \quad x_{n_0+k} (0,\vz_\B; f) = \E_{-k} \psi \quad
\text{for} \quad 0 \leq k \leq k_1 .
\end{eqnarray}
Indeed, consider $f (\cdot)$ defined by
\begin{eqnarray}
f (j) & = & \vz_\X   \ \qquad \text{ for } \  0 \leq j \leq n_0-2 , \quad \label{e f1} \\
f (n_0-1) & = & \psi ,  \quad
\label{e f12}\\
f(n_0+k) & = & - L (n_0+k,k) \psi
% \ ( \, = - L (n_0+k) \E_{-k} \psi \, )
 \qquad \text{ for } \ \quad 0 \leq k \leq k_1-1,
 \label{e f2}\\
f(n_0+k) & = & \vz_\X  \ \qquad \text{ for } \ \quad k \geq k_1 \label{e f3}.
\end{eqnarray}
Then it is easy to see that (\ref{e xn=psi}) holds.
Note that %(\ref{e Lnm})
\[
x (n_0+k_1+1) = L (n_0+k_1,k_1) \psi  \quad \text{ and } \quad
|x (n_0+k_1+1)| \leq \| x(\cdot , 0, \vz_\B; f) \|_q.
\]
Applying Proposition \ref{p GaL} to $f$ defined by (\ref{e f1})--(\ref{e f3}), we get
\begin{eqnarray*}
| L (n_0+k_1,k_1) \psi | \leq \| x(\cdot , 0, \vz_\B; f) \|_q \leq K_{p,q,L}  \| f \|_p .
\end{eqnarray*}
It follows from (\ref{e ||q<||p}) that
\begin{eqnarray*}
\| f \|_p \leq \| f \|_1 = |\psi| + \sum_{k=0}^{k_1-1} | L (n_0+k,k) \psi |
\leq |\psi| \left( 1+ \sum_{k=0}^{k_1-1} | L (n_0+k,k)  |_{\X \to \X} \right).
\end{eqnarray*}
Combining the last two inequalities with (\ref{e L(n+k,k)}) for $0\leq k \leq k_1 -1$ and
the fact that $K_{p,q,L} \geq 1$, we get
\begin{eqnarray*}
\| L (n_0+k_1,k_1) \|_{\X \to \X} \leq
K_{p,q,L} \left( 1+ \sum_{k=0}^{k_1-1} 2^k K_{p,q,L}^{k+1} \right) \leq
K_{p,q,L}^{k_1+1} \left( 1 + \sum_{k=0}^{k_1-1} 2^k \right) =
2^{k_1} K_{p,q,L}^{k_1+1}.
\end{eqnarray*}
This completes the proof of \textbf{(i)}.

\textbf{(ii)} The assertion
%$\| L(\cdot) \Pr_{[j,0]} \|_\infty < \infty $
$\sup\limits_{n \in \Z^+} \| L(n) \Pr_{[j,0]} \|_{\B^\ga \to \X} < \infty$ follows from the estimate
\begin{eqnarray} \label{e supLP}
\sup\limits_{n \geq -j+1} \ \| L(n) \Pr_{[j,0]} \|_{\B^\ga \to \X} \ \leq \
 K_{p,q,L} \frac{\left(2e^\ga K_{p,q,L} \right)^{-j+1} -1}{2e^\ga K_{p,q,L} -1} \  .
\end{eqnarray}
Let us show that statement (i) implies (\ref{e supLP}). Note that for $n \geq -j+1$
and
$j \leq m \leq 0$,
\[
| L(n) \Pr_{\{m\}} \vphi | = | L(n,-m) \vphi^{[m]} | \leq
2^{-m} K_{p,q,L}^{-m+1} | \vphi^{[m]} | \leq 2^{-m} K_{p,q,L}^{-m+1} e^{-m\ga} | \vphi |_{\B^{\ga}} \ .
\]
Hence,
\[
| L(n) \Pr_{[j,0]} \vphi | \leq \sum_{m=j}^{0}
| L(n) \Pr_{\{m\}} \vphi | \leq | \vphi |_{\B^{\ga}} \sum_{m=j}^{0} 2^{-m} e^{-m\ga} K_{p,q,L}^{-m+1}
\leq | \vphi |_{\B^{\ga}} K_{p,q,L}  \frac{\left(2e^\ga K_{p,q,L} \right)^{-j+1} -1}{2e^\ga K_{p,q,L} -1}
 \ .
\]
\end{proof}

\subsection{Proof of Theorem \ref{t UES} \label{ss proofUES}}

The proof is based on Theorem \ref{t ordUES} and
the reduction of (\ref{e nh}) to (\ref{e ndnh}).
The facts that \textbf{(i) $\Leftrightarrow$ (ii)} and that
\textbf{(i)} and  \textbf{(ii)} are equivalent to
UE stability of (\ref{e ndh}) are established in  Proposition \ref{p UES}.

Let us prove that \textbf{ (iii) implies UE stability of (\ref{e ndh})}.
Taking into account Theorem \ref{t ordUES}, it is enough to prove that \textbf{(iii)} implies
\begin{equation} \label{e ylq}
\| y(\cdot , 0, \vz_\B; g) \|_q \leq C_1 \| g \|_p \ , \qquad g \in \l^p (\B^\ga),
\end{equation}
with a certain constant $C_1 = C_1 (L)>0$.

By Proposition \ref{p rep},
\begin{eqnarray*}
\| y (\cdot,0,\vz_\B ; g) \|_q \leq
\| x_\bullet (0,\vz_\B; g^0 + L \h) \|_q + \| \h \|_q .
\end{eqnarray*}
Applying Proposition \ref{p lp} to the first term (note that $x_{0} = \vz_\B$) and
inequality (\ref{e ||q<||p}) to the second, we get
\begin{eqnarray*}
\| y (\cdot,0,\vz_\B ; g) \|_q \leq
C_2 (q) \, \| x (\cdot, 0 ,\vz_\B; g^0 + L \h) \|_q + \| \h \|_p ,
\end{eqnarray*}
where $C_2 (q) := (1 - e^{-q\ga})^{-1/q}$ for $q<\infty$ and $C_2 (\infty) :=1$.
From $(\l^p,\l^q)$-stability and Proposition \ref{p GaL} we obtain
\begin{eqnarray*}
\| y (\cdot,0,\vz_\B ; g) \|_q \leq
 C_2 (q) \, K_{p,q,L} \, \|g^0 \|_p + C_2 (q) \, K_{p,q,L} \, \| L \h \|_p + \| \h  \|_p \ .
\end{eqnarray*}
Note that (\ref{e LPlinf}) and Proposition \ref{p |Lnk|<} (ii) imply $\| L ( \cdot ) \|_\infty <
\infty$.
This and Proposition \ref{p hlp} yield
\begin{eqnarray*}
\| y (\cdot,0,\vz_\B; g) \|_q & \leq & C_2 (q) \, K_{p,q,L} \,  \| g^0 \|_p +
C_2 (q) \, K_{p,q,L} \, \| L \|_\infty \| \h \|_p + \| \h \|_p
\leq \\ & \leq &
C_2 (q) \, K_{p,q,L} \, \| g \|_p +
\bigl[ C_2 (q) \, K_{p,q,L} \, \| L \|_\infty +1 \bigr] (1-e^{-\ga})^{-1} \|  g \|_p \ .
\end{eqnarray*}
This completes the proof of (\ref{e ylq}).

Let us show that \textbf{UE stability of (\ref{e ndh}) implies (iii).}
It follows from Theorem \ref{t ordUES} and formula (\ref{e xn=yn}) that UE stability of (\ref{e ndh})
implies
$x_\bullet (0, \vz_\B ; f) \in \l^q (\X)$ for any $f \in \l^p (\X)$. So system (\ref{e nh}) is
$(\l^p,\l^q)$-stable.

Finally, note that UE stability of (\ref{e ndh}) implies $\| L (\cdot)\|_\infty < \infty$ and so
implies (\ref{e LPlinf}) for  every $m \in \Z^-$. Indeed,
we see from Remark \ref{r Alinf} (1) that \textbf{(ii)} yields
$\| D (\cdot) \|_\infty < \infty$ (the operators $D(n)$ are defined by (\ref{e D})). Now (\ref{e |D|}) 
implies $\| L (\cdot) \|_\infty < \infty$.
This completes the proof.

\section{Uniform stability and $(\l^1,\l^\infty)$-stability}

In this section we prove that in the case when $p=1$ and $q = \infty$, the theorem analogous to
Theorem \ref{t UES} is valid with uniform stability instead of UE-stability. We also consider
the more difficult case when $\ga=0$. As before, our method is based on the reduction
of system (\ref{e nh}) to the  first order system (\ref{e ordnh}).

\subsection{Preliminaries} \label{ss preUS}

Let a function $L: \Z^+ \to \L (\B, \X)$ define system (\ref{e nh}) on a seminormed phase space $\B$.
Then homogeneous system (\ref{e h}) is called \emph{uniformly stable (US, in short) in $\X$ w.r.t.
the phase space $\B$} if (\ref{e ExStX}) holds with $\nu = 0$.
(If the phase space $\B$ is fixed we will say
in brief that a system is US in $\X$.)

\emph{Uniform stability in $\B$} of (\ref{e h}) and \emph{uniform stability} of the  first order
system (\ref{e ordh})
are defined in the similar way placing $\nu = 0$ in  Definitions \ref{d UESn} (2) and \ref{d ordUES} (1),
respectively.

We will use the following result concerning first order systems.

\begin{theorem}[cf. \cite{AVM96}] \label{t ordUS}
  Let a function $\A : \Z^+ \to \L (\W)$ define the first order system (\ref{e ordnh}).
  Then the corresponding homogeneous system (\ref{e ordh}) is US if and only if system (\ref{e ordnh})
  is $(\l^1,\l^\infty)$-stable.
\end{theorem}

\begin{remark} \label{r USAinf<inf}
This result was obtained in \cite[Theorem 6]{AVM96}, formally under the additional assumption
$\sup_{n \in \Z^+} \| \A (n) \|_{\W \to \W} < \infty$. This additional assumption  can be easily dropped
in the way shown in Remark \ref{r Alinf}. In fact, uniform stability of (\ref{e ordh}), as well as
$(\l^1,\l^\infty)$-stability of (\ref{e ordnh}), implies
$\sup_{n \in \Z^+} \| \A (n) \|_{\W \to \W} < \infty$.
\end{remark}

\begin{proposition} \label{p US}
Let $\ga \geq 0$ and let $L: \Z^+ \to \L (\B^{\ga}, \X)$. Then the following statements are equivalent:
\begin{description}
\item[(i)] System (\ref{e h}) is US in $\X$ w.r.t.  $\B^{\ga}$.
\item[(ii)] System (\ref{e h}) is US in $\B^{\ga}$.
\item[(iii)] System (\ref{e ndh}) is US.
\end{description}
\end{proposition}

The proof is the same as that of Proposition \ref{p UES} if we set $\nu=\nu_1=0$.

\begin{proposition} \label{p US to Linf}
Let $\ga \geq 0$ and let $L: \Z^+ \to \L (\B^{\ga}, \X)$. Then any of assertions (i)-(iii) of
Proposition \ref{p US} implies $\sup_{n \in \Z^+} \| L (n) \|_{\B^\ga \to \X} < \infty$.
\end{proposition}

The proposition follows from assertion (iii) of Proposition \ref{p US}, Remark \ref{r USAinf<inf},
and formula (\ref{e |D|}).

\subsection{Uniform stability in the phase space $\B^\ga$ with $\ga >0$. \label{ss USres ga>0}}

\begin{theorem} \label{t US}
Let $\ga > 0$ and let a function $L: \Z^+ \to \L (\B^{\ga}, \X)$ define system (\ref{e nh}).
Then the following statements are equivalent
\begin{description}
\item[(i)] System (\ref{e h}) is US in $\B^{\ga}$ (or, equivalently, in $\X$ w.r.t. $\B^{\ga}$).
%\item[(ii)] System (\ref{e h}) is US in $\B^{\ga}$.
\item[(ii)] System (\ref{e nh}) is $(\l^1,\l^\infty)$-stable and condition
    (\ref{e LPlinf}) is fulfilled.
\end{description}
\end{theorem}

%\begin{proof}[Proof of Theorem \ref{t UE}.]
The proof of Theorem \ref{t US} is similar to the proof of Theorem \ref{t UES} with the use of
Proposition \ref{p US} instead of Proposition \ref{p UES}, Remark \ref{r lp} instead of 
Proposition \ref{p lp}, Theorem \ref{t ordUS} instead of
Theorem \ref{t ordUES},
and Remark \ref{r USAinf<inf} instead of Remark \ref{r Alinf} (1).
%\end{proof}

\begin{remark} \label{r US cond repl}
 Example \ref{ex Pmn} and Proposition \ref{p US to Linf} show that condition (\ref{e LPlinf})
 in Theorem \ref{t US} \emph{cannot be replaced} by (\ref{e LPmn})
with $\{ m_n \}_1^\infty $ such that
$\lim m_n = -\infty$.
\end{remark}

Recall that systems with bounded delay are defined in Section \ref{ss UESres}.

\begin{corollary} \label{c USbd}
 Let $\ga \in \R$.
If (\ref{e nh}) is a  system with bounded delay of order $d$, then the following statements are equivalent:
\begin{description}
\item[(i)] System (\ref{e h}) is US in $\X$ with respect to $\B^\ga$
(or, equivalently, w.r.t. $\B_{[-d+1,0]}^\ga$).
\item[(ii)] System (\ref{e nh}) is $(\l^1,\l^\infty)$-stable.
\end{description}
\end{corollary}

The proof is similar to that of Corollary \ref{c UESbd}.

%\begin{remark}
%Under the additional assumption $\sup_{n \in \Z^+} \| L(n) \|_{\B^0 \to \X}$,
%Corollary \ref{c USbd} for subdiagonal systems
%(see Definition \ref{d tr})
%with bounded delay can be derived from \cite[Theorem 1]{BB_JMAA05}.
%\end{remark}

\subsection{The case $\ga = 0$ \label{ss USres ga=0}}
%\begin{remark} \label{r US ga=0 repl}

Example \ref{ex US ga=0 repl} shows that in general Theorem \ref{t US} is not valid for
$\ga=0$.
%The reason why the proof of Theorem \ref{t US} does not work for $\ga=0$ is that
%the estimate of Proposition \ref{p hlp} degenerates.
In this subsection we give several results on uniform stability of (\ref{e h}) in $\B^0$.

Let $L: \Z^+ \to \L (\B^{0}, \X)$.
Recall that the operator $H: \Ssp_+ (\B^0) \to \Ssp_+ (\B^0)$ is defined in (\ref{e reph})
%\[
%H g := \h , \quad \text{where} \ \h \ \text{is defined by (\ref{e reph}}),
%\] and,
Let us define the operator $M_L: \Ssp_+ (\B^0) \to \Ssp_+ (\X) $ by
\[
(M_L g) (n) := L(n) g(n), \quad n \in \Z^+
, \quad g \in \Ssp_+ (\B^0)
.
\]
As in the proof of Proposition \ref{p GaL}, $\Gamma$ is the linear operator from
$\Ssp_+ (\X)$ to $\Ssp_+ (\X)$ defined by
\[
\Gamma : f \mapsto x(\cdot , 0, \vz_\B; f).
\]
In the following theorem we use the product $\Gamma M_L H$ of
the three above defined operators and the image $\left( \Gamma M_L H \right) \, \l^1 (\B^0)$ of
the space $\l^1 (\Z^+, \B^0)$ under the operator $\Gamma M_L H $. The image is understood in 
the usual sense
\[
\left( \Gamma M_L H \right) \, \l^1 (\B^0) :=
\{ \ x (\cdot) \in \Ssp_+ (\X) \ : \ x = \Gamma M_L H g \ \ \text{for some} \ g \in \l^1 (\B^0) \} .
\]

\begin{theorem} \label{t UEga=0}
Let a function $L: \Z^+ \to \L (\B^{0}, \X)$ define system (\ref{e nh}).
Then the following statements are equivalent
\begin{description}
\item[(i)] System (\ref{e h}) is US in $\B^{0}$ (or, equivalently, in $\X$ w.r.t. $\B^{0}$).
\item[(ii)] System  (\ref{e nh}) is $(\l^1,\l^\infty)$-stable and
$\left( \Gamma M_L H \right) \, \l^1 (\B^0) \subseteq \l^{\infty} (\X)$.
\end{description}
\end{theorem}

\begin{proof}
\textbf{(i) $\Rightarrow$ (ii).} By Proposition \ref{p US} and Theorem \ref{t ordUS},
assertion (i) implies that the first order system (\ref{e ndnh}) considered in the space $\B^{0}$
is $(\l^1,\l^\infty)$-stable. This and (\ref{e xn=yn}) yield $(\l^1,\l^\infty)$-stability of (\ref{e nh}).
In particular,
\begin{equation} \label{e Gg0}
\Gamma g^{[0]} \in \l^\infty (\X) \quad \text{for any} \quad g \in \l^1 (\B^0) .
\end{equation}

For $g \in \l^1 (\B^0)$, equality (\ref{e rep1}),  $(\l^1,\l^\infty)$-stability of (\ref{e ndnh}),
and
$|x (n,0,\vz_\B;f)| \leq | x_n (0,\vz_\B;f) |_{\B^0}$ imply
\[
\Gamma (g^{[0]} + M_L H g) = x (\cdot,0,\vz_\B;g^{[0]} + M_L H g) \in \l^\infty (\X).
\]
Taking into account (\ref{e Gg0}), we get
$\Gamma M_L H g \in \l^{\infty} (\X)$.

Inverting the above arguments with the use of Remark \ref{r lp} one can get \textbf{(ii) $\Rightarrow$ (i).}
\end{proof}

The condition $\left( \Gamma M_L H \right) \, \l^1 (\B^0) \subseteq \l^{\infty} (\X)$ 
has a complicated nature.
The following sufficient conditions are simpler for understanding.

\begin{corollary} \label{c UEga=0}
Let a function $L: \Z^+ \to \L (\B^{0}, \X)$ define system (\ref{e nh}).
Then each of the following two conditions is sufficient for (\ref{e h}) to be US in $\B^{0}$:
\begin{description}
\item[(i)] System  (\ref{e nh}) is $(\l^\infty,\l^\infty)$-stable and
\begin{equation} \label{e LPlinf ga=0}
\text{there exists } \ m \in \Z^- \ \ \text{such that} \quad
\| L (\cdot) \Pr_{[-\infty,m]} \|_\infty :=
\sup_{n \in \Z^+} \| L (n) \Pr_{[-\infty,m]} \|_{\B^{0} \to \X} < \infty \ .
\end{equation}
\item[(ii)] System  (\ref{e nh}) is $(\l^1,\l^\infty)$-stable and
\begin{equation} \label{e LPl1}
\text{there exists } \ m \in \Z^- \ \ \text{such that} \quad
 \| L (\cdot) \Pr_{[-\infty,m]} \|_1 :=
 \sum_{n=0}^{\infty} \| L (n) \Pr_{[-\infty,m]} \|_{\B^{0} \to \X} < \infty \ .
\end{equation}
\end{description}
\end{corollary}

\begin{proof}
In the both cases we check assertion (ii) of Theorem \ref{t UEga=0} using the following estimates
\begin{eqnarray} \label{e H inf to 1}
\| H g \|_\infty := \sup_{n \in \Z^+} | (H g) (n) |_{\B^0} & \leq & \| g \|_1 , \\
\| \; \Pr_{[j,0]} \; (H g) (\cdot) \; \|_1 :=
\sum_{n=0}^{\infty} \left| \Pr_{[j,0]} (H g)(n) \right|_{B^0} & \leq & (-j) \; \| g \|_1 ,
\quad j \in \Z^- . \label{e PjH1 to 1}
\end{eqnarray}
The first one follows immediately from (\ref{e hm=sumg}). For the proof of the second, let us note that
(\ref{e hm=sumg}) implies
\[
\left| \Pr_{[j,0]} (H g)(n) \right|_{B^0} = \max_{j \leq m \leq 0} \left| \, (H g)^{[m]} (n) \, \right|
\leq \sum_{k=\max \{0, n+j\} }^{n-1}  |g (k)|_{\B^0} .
\]
Hence (recall the sum convention (\ref{e sumconv}))
\[
\sum_{n=0}^{\infty} \left| \Pr_{[j,0]} (H g)(n) \right|_{B^0} \leq
\sum_{n=1}^{\infty} \ \ \sum_{k=\max \{0, n+j\} }^{n-1}  |g (k)|_{\B^0} .
\]
Since each $ g (n) $ participates in the last sum
at most $(-j)$ times, we get (\ref{e PjH1 to 1}).

\textbf{(i) implies uniform stability of (\ref{e h}).}
 $(\l^1,\l^\infty)$-stability follows from  $(\l^\infty,\l^\infty)$-stability.

Formulae (\ref{e supLP<inf}) and (\ref{e LPlinf ga=0}) imply that
$\| L \|_\infty := \sup_{n \in \Z^+} \| L (n) \|_{\B^{0} \to \X}< \infty$.
From  (\ref{e H inf to 1}) we get
\[
\| M_L H g \|_\infty \leq \| L \|_\infty \| H g \|_\infty \leq \| L \|_\infty \| g \|_1 .
\]
So $M_L H g \in \l^\infty (\X)$ for $g \in \l^1 (\B^0)$ and therefore $(\l^\infty,\l^\infty)$-stability
implies $\Gamma M_L H g \in \l^{\infty} (\X)$.

\textbf{(ii) implies uniform stability of (\ref{e h}).}
Let $g \in \l^1 (\B^0)$.
Without loss of generality, we can assume $m \leq -1$ in
(\ref{e LPl1}).
Taking $j=m+1$,   one can get from (\ref{e PjH1 to 1}) and
(\ref{e supLP<inf}) that
\begin{eqnarray}
\sum_{n=0}^{\infty} \left| L (n) \Pr_{[m+1,0]} (H g) (n) \right| & \leq &
\| \,  L (\cdot) \Pr_{[m+1,0]} \, \|_\infty \| \, \Pr_{[m+1,0]} (H g) (\cdot)  \, \|_1 \leq \notag
\\
& \leq &
 (-m-1) \; \| L (\cdot) \Pr_{[m+1,0]} \|_\infty  \| g \|_1 <\infty . \label{e LPm0Hg l1}
\end{eqnarray}
From (\ref{e LPl1}) and (\ref{e H inf to 1}) we get
\begin{eqnarray}
\sum_{n=0}^{\infty} \left| L (n) \Pr_{[-\infty,m]} (H g) (n) \right| \leq
\| L (\cdot) \Pr_{[-\infty,m]} \|_1 \; \| H g \|_\infty \leq
\| L (\cdot) \Pr_{[-\infty,m]} \|_1 \; \| g \|_1 < \infty.
\label{e LPimHg l1}
\end{eqnarray}
Combining (\ref{e LPm0Hg l1}) and (\ref{e LPimHg l1}), we see that $M_L H g \in \l^1 (\X)$ for every
$g \in \l^1 (\B^0)$. From $(\l^1,\l^\infty)$-stability we get $\Gamma M_L H g \in \l^{\infty} (\X)$.
\end{proof}

\begin{remark}
Example \ref{ex LPlp} shows that condition (\ref{e LPl1}) in
%assertion (ii) of
Corollary \ref{c UEga=0} cannot be relaxed to condition $\| L (\cdot) \Pr_{[-\infty,-1]} \|_p < \infty$
with $p>1$.
\end{remark}

\section{Subdiagonal systems and independence of stability properties of the parameter $\ga$
\label{s Ind ga}}

The continuous embedding
\[
\B^\ga \subseteq \B^\delta , \quad
%\text{and} \quad
| \vphi |_{\B^\delta} \leq | \vphi |_{\B^\ga} \qquad
\text{for} \quad -\infty < \ga \leq \delta < \infty ,
\]
and the definitions of uniform stability and UE stability imply easily the following statement.

\begin{proposition} \label{p ga<delta}
Let $\delta \in \R$. If system (\ref{e h}) is UES (US) in $\X$ w.r.t. $\B^\delta$,
then for any $\ga \in (-\infty,\delta)$ system (\ref{e h}) is UES (resp., US) in $\X$ w.r.t. $\B^\ga$.
\end{proposition}

It is easy to see that the inverse implication is not true (for instance, from Example \ref{ex ga<delta}).

In this section we will show that under assumption (\ref{e LPlinf}) with $\ga>0$,
uniform stability and UE stability in $\B^\delta$ do not depend on the choice of the parameter
$\delta \in (0,\ga]$, and, moreover, do not depend on the 'upper-triangular' part of
the infinite operator-matrix $ \bigl( \, L (n,j) \, \bigr)_{n,j \geq 0}$.%_{n,j \in \Z^+}$.

\begin{definition} \label{d tr}
Assume that $L(n)$ is defined on $\B_{\fin}$ for all $n \in \Z^+$ and that
$L (n,k) \in \L (\X,\X)$ for all $n,k \in \Z^+$.
Then we will say that (\ref{e nh}) is \emph{a subdiagonal system} if
\begin{eqnarray} \label{e Lupt=0}
 L (n) \Pr_{[-\infty,-n]} \vphi = \vz_\X , \quad \vphi \in \dom L (n), \quad n \in \Z^+ .
\end{eqnarray}
\end{definition}

If (\ref{e nh}) is a subdiagonal system, then it
can be written in the form
\begin{eqnarray} \label{e nhtr}
x(n+1) = \sum_{j = 1}^n L (n,n-j) x (j) + f(n), \ n \in \Z^+ , \
(x(0)= \vz_\X \text{ due to  (\ref{e sumconv}})
\end{eqnarray}
 and it can be considered on the whole vector space $\X^{\Z^-}$.

\begin{remark}
One can see that subdiagonal system is a particular case of Volterra difference system with unbounded 
delay (see (\ref{e Volunb})).  Subdiagonal systems are characterized by the condition $L(n,n) = 0$.
\end{remark}

Having an arbitrary operators $L(n): \dom  L(n) \to \X$, $n \in \Z^+$, with domains
satisfying $\B_\fin \subseteq \dom  L(n) \subseteq \X^{\Z^-} $,
we define the operator-valued function $ L_\subd $  on $\Z^+$ by
\begin{eqnarray} \label{e trL0}
 L_\subd (0) \vphi & = & \vz_\X , \qquad  \vphi \in \X^{\Z^-}, \\
 L_\subd (n) \vphi & = &  L (n) \Pr_{[-n+1,0]} \vphi ,
\quad \vphi \in \X^{\Z^-}, \quad n \in \N .\label{e trLn}
\end{eqnarray}
If $L(n,k) \in \L (\X,\X)$ for all $n,k \in \Z^+$, then $ L_\subd$ defines a subdiagonal system. 
For a function $ L:\Z^+ \to \L (\B^\ga,\X)$ the assumption $L(n,k) \in \L (\X,\X)$
is always fulfilled and therefore the following definition is natural.

\begin{definition} \label{d subdas}
Let $\ga \in \R$ and let a function $ L:\Z^+ \to \L (\B^\ga,\X)$ define system
(\ref{e nh}). Then we will say that the system defined by the function $L_\subd$ is
the subdiagonal system associated with system (\ref{e nh}).
\end{definition}

Definition \ref{d subdas} is justified by the following proposition.

\begin{proposition} \label{p tr(lp,lq)}
In the settings of Definition \ref{d subdas} system (\ref{e nh}) is $(\l^p,\l^q)$-stable 
if and only if
the associated subdiagonal system  is
$(\l^p,\l^q)$-stable.
\end{proposition}

For the proof it is enough to notice that Definition \ref{d lplq st} implies that
$(\l^p,\l^q)$-stability  of system (\ref{e nh}) does not depend on the parts
$L (n) \Pr_{[-\infty,-n]}$ of the operators $L(n)$.

\begin{theorem} \label{t ga ind}
Let $\ga >0$ and let a function $ L:\Z^+ \to \L (\B^\ga,X)$ define system (\ref{e nh}).
Assume that (\ref{e LPlinf}) holds.
Then the following statements are equivalent:
\begin{description}
\item[(i)] System (\ref{e nh}) is UES (US) in $B^\ga$.
\item[(ii)] The subdiagonal system associated with  (\ref{e nh}) is UES (resp, US) in $B^\ga$.
\item[(iii)] For any $\delta >0$ and any function $\wt L : \Z^+ \to \L(\B^{\delta}, \X)$ such that
\begin{eqnarray}
& \wt L_\subd (n) = L_\subd (n) ,  \quad n \in \Z^+, \quad \text{and} \notag
\\
%\text{there exists } \ l \in \Z^- \ \ \text{such that} \quad
& \sup_{n \in \Z^+} \| \wt L (n) \Pr_{[-\infty, l]} \|_{ \B^{\scriptstyle \delta} \to \X}
< \infty  \quad \text{for a certain} \quad l \in \Z^-,
\label{e tLPlinf}
\end{eqnarray}
the system defined by $\wt L (\cdot)$
is UES (resp., US) in $\B^{\delta}$.
\end{description}
\end{theorem}

%\begin{proof}
The theorem follows immediately from Proposition \ref{p tr(lp,lq)} and
%equivalence (ii) $\Leftrightarrow$ (iii) of
Theorem \ref{t UES} (resp., Theorem \ref{t US}).
%\end{proof}

\begin{remark}
\textbf{1)} Taking the function $L$ defined in Example \ref{ex US ga=0 repl},
one can see that the system defined by $L_\subd$ is
US in any of $\B^\ga$ with $\ga > 0$, while the system defined by $L$ is not US in $\B^\ga$, $\ga > 0$.
So Example \ref{ex US ga=0 repl} shows that neither condition (\ref{e LPlinf}) nor (\ref{e tLPlinf})
can be dropped in the US version of Theorem \ref{t ga ind}.
For the UES case of Theorem \ref{t ga ind} the same conclusion can be inferred  from
Example \ref{ex BB05 ex2}.

\textbf{2)} The systems introduced in Examples \ref{ex US ga=0 repl} and
 \ref{ex BB05 ex2} satisfy condition (\ref{e LPlinf}) with $\ga \leq 0$.
This  implies easily that in general Theorem \ref{t ga ind} is not valid for $\ga \leq 0$ in both the
$US$ and $UES$ (in $\X$) versions.
\end{remark}

\section{Examples  \label{s ex}}

In this section we assume that the Banach space $\X $ is nontrivial, i.e., $\X \neq \{ \vz_\X \}$.
In particular, all the arguments below are valid if $\X = \R$
or $\X = \C$.

The following example
shows that condition (\ref{e LPlinf}) cannot be omitted
in Theorem \ref{t UES} and Corollary \ref{c lplq}.

\begin{example} \label{ex lplq ind}
Let $1 \leq q < r  \leq \infty$, $a (\cdot) \in l^{r} (\R) \setminus l^{q} (\R)$.
Let us define $L(n):\X^{\Z^-} \to \X$ by $L(0) = 0$ and $L(n) \vphi  =  a(n) \vphi^{[-n+1]}$ for
$n \in \N$.
%\begin{eqnarray*}
%L(0) = 0 , \qquad
%L(n) \vphi & = & a(n) \vphi^{[-n+1]}, \quad
%\vphi \in \X^{\Z^-},
%\quad n \in \N .
% \label{e L=a lplq}
%\end{eqnarray*}
The function $L$ defines the system
\begin{eqnarray} \label{e eq exlplq}
x (1) = f(0), \qquad x(n+1) = a(n) x (1) + f(n) , \quad n \in \N ,
\end{eqnarray}
on all the spaces $\B^\ga$ with $\ga \in \R$.
It is easy to see that for any $p \leq r$
system (\ref{e eq exlplq}) is $(\l^p,\l^r)$-stable, but it \emph{is not $(\l^p,\l^q)$-stable}.

Indeed, considering the solution $x (n) = x (n,0,\vz_{\B};f) $ with $f \in \l^p$, we get
$ x (n+1) = a(n) f(0) + f(n)$,  $ n \in \N$.
Since $ a (\cdot) \in \l^{r} $ and $f (\cdot) \in \l^p \subseteq \l^{r}$,
we see that $x( \cdot) \in \l^{r}$.
On the other hand, it follows from $ a (\cdot) \notin \l^{q} $ that $x( \cdot) \notin \l^{q}$
whenever $f(0) \neq 0$ and $f \in \l^{p}  \cap \l^q$.

Note that
\[
x (n+1,1,\vphi; \fz) = a (n) x (1,1,\vphi; \fz) = a (n) \vphi^{[0]}, \quad n \in \N .
\]
Hence, $a (\cdot) \notin l^{q}$ implies that the homogeneous system associated with 
(\ref{e eq exlplq}) \emph{is not UES in $\X$
w.r.t. any of the spaces $\B^\ga$, $\ga \in \R$}.
\end{example}

\vspace{2mm}

The next example shows that
condition (\ref{e LPlinf}) cannot be replaced by the
less restrictive condition (\ref{e LPmn}) in Theorems \ref{t UES} and \ref{t US}.

\begin{example} \label{ex Pmn}
Let $1 \leq p \leq q \leq \infty$, $\ga > 0$, and let $\B^{\ga}$ be the phase space.
Let a sequence $m_n \in \Z^-$, $n\in \Z^+$, be such that
$\liminf m_n = -\infty$. Then there exists $L:\Z^+ \to \L (\B^{\ga}, \X)$ such that:
\item[(i)] condition (\ref{e LPmn}) is fulfilled,
\item[(ii)] system (\ref{e nh}) is $(\l^p,\l^q)$-stable,
\item[(iii)] but $\sup_{n \in \Z^+} \| L (n) \|_{\B^\ga \to \X} = \infty$ and so system (\ref{e h})
\emph{is neither UES nor US in $\X$} (due to Remark \ref{r |L|<inf} and Proposition \ref{p US to Linf}).

Let us construct the corresponding operator-function. Since $\liminf m_n = -\infty$, we can choose an
increasing sequence $n_k \in \Z^+$, $k \in \N$, such that
\begin{eqnarray}
n_k & > & n_{k-1} + k +1 , \label{e nk nk+1}\\
m_{n_k} & < & - k . \label{e mnk<-k}
\end{eqnarray}
Let us define $L(\cdot)$ by
\begin{eqnarray}
L(n) & = & 0 \quad \text{if} \quad n \not \in \{ n_k \}_1^{\infty} , \label{e Ln n neq nk}
 \\
L(n_k) \vphi & = & \vphi^{[-k]} , \quad \vphi \in \B^\ga, \quad k \in \N . \label{e Lnk}
\end{eqnarray}

Then condition (\ref{e LPmn}) is fulfilled since (\ref{e Ln n neq nk}), (\ref{e mnk<-k}) and (\ref{e Lnk})
imply $L (n) \Pr_{[-\infty,m_n]} = 0$.

Further, $\sup_{n \in \Z^+} \| L (n) \|_{\B^\ga \to \X} = \infty$. Indeed, (\ref{e Lnk}) yields
\[
\| L(n_k) \|_{\B^\ga \to \X} =
\sup_{\vphi \neq 0_\B} \frac{| \vphi^{[-k]} |}{|\vphi|_{\B^\ga}} = e^{k\ga} , \qquad
\text{and therefore} \qquad \lim_{k \to \infty} \| L(n_k) \|_{\B^\ga \to \X} =  \infty .
\]
%So $\lim_{k \to \infty} \| L(n_k) \|_{\B^\ga \to \X} =  \infty$.

On the other hand, system (\ref{e nh}) is $(\l^p,\l^q)$-stable. Indeed, due to $p \leq q$ and
(\ref{e ||q<||p}),
it is enough to show that $x (\cdot,0,\vz_\X;f) \in \l^p$ for any $f (\cdot) \in \l^p$.
It is easy to see that
\begin{eqnarray}
x (n+1,0,\vz_\X;f) & = & f (n) \qquad \text{if} \quad n \not \in \{ n_k \}_1^{\infty} ,
\label{e xn n not nk}
 \\
x (n+1,0,\vz_\X;f) & = & x_{n_k}^{[-k]} + f(n_k) \quad \text{if} \quad n = n_k . \notag %\label{e }
\end{eqnarray}
By definition, $x_{n_k}^{[-k]} = x (n_k-k)$. Taking into account (\ref{e nk nk+1}),
we see that $n_k - k > n_{k-1}+1 $ and so (\ref{e xn n not nk}) implies $x_{n_k}^{[-k]} = f (n_k-k-1)$.
Hence,
\begin{eqnarray}
x (n+1,0,\vz_\X;f) = f (n_k-k-1) + f(n_k) \quad \text{if} \quad n = n_k . \label{e xnk+1}
\end{eqnarray}
By (\ref{e nk nk+1}) both sequences $\{n_k\}$ and $\{n_k-k-1\}$ are strictly increasing.
Thus, by Minkowski's inequality,
\begin{eqnarray*}
\| x (\cdot ,0,\vz_\X;f)\|_p \leq 2 \| f \|_p. \label{e |x|<2|f|}
\end{eqnarray*}
%Indeed, in the case $p=\infty$, (\ref{e |x|<2|f|}) immediately follows from
%(\ref{e xn n not nk}) and (\ref{e xnk+1}).
%If $p<\infty$, we have
%\[
%\left( \sum_{n=0}^\infty |x (n,0,\vz_\X;f)|^p \right)^{1/p}=
%\left( \sum_{n=0}^\infty |f(n)|^p \right)^{1/p} +
% \left( \sum_{k = 1}^\infty |f (n_k-k-1)|^p \right)^{1/p}.
% \]
% Now (\ref{e nk nk+1}) yields that each of $f(n)$ appears in the last sum at most once.
% So (\ref{e |x|<2|f|}) holds.
\end{example}

\vspace{2mm}

\cite[Example 2]{BB_JMAA05} can be used to prove that neither condition (\ref{e LPlinf}) 
nor condition (\ref{e tLPlinf})
can be dropped in the UES version of Theorem \ref{t ga ind} and that for $\ga \leq 0$ the 
implication (iii)$\Rightarrow$(i) of Theorem \ref{t UES} is not valid in general.

\begin{example}[\cite{BB_JMAA05}] \label{ex BB05 ex2}
Define  $L(n)$ by $L(n) \vphi = \frac 12 \vphi^{[0]}+ \vphi^{[-n]}$, $n \in \Z^+$.
%\begin{eqnarray*}
% L(n) \vphi = \frac 12 \vphi^{[0]}+ \vphi^{[-n]},
% \quad \quad
%n \in \Z^+ .
% \label{e }
%\end{eqnarray*}
The function $L$ defines the system
\begin{eqnarray} \label{e eqBB05ex}
x(n+1) = \frac 12 x(n)+ x (0) + f(n) , \quad \quad n \in \Z^+ .
\end{eqnarray}
%on all the spaces $\B^\ga$ with $\ga \in \R$.

%\textbf{1)}
Let us show that for any $\ga \leq 0$:
\item[(i)] function $L$ satisfies condition (\ref{e LPlinf}),
\item[(ii)] system (\ref{e eqBB05ex}) is $(\l^p,\l^q)$-stable for any $1\leq p \leq q \leq \infty$,
\item[(iii)] but the homogeneous system associated with (\ref{e eqBB05ex}) 
\emph{is not UES in $\X$ w.r.t. $\B^\ga$}.

Indeed,
(\ref{e LPlinf}) is satisfied in its strongest form (i.e., with $m =0$) since for $\ga \leq 0$
\[
| L (n) \vphi | \leq \frac 12 |\vphi^{[0]}| + |\vphi^{[-n]}| \leq 
\left( \frac 12 + e^{n\ga} \right) |\vphi |_{\B^\ga} \leq \frac 3 2 |\vphi |_{\B^\ga} .
\]
By induction,
\[
x (n , 0, \vz_\B; f ) = \sum_{k=0}^{n-1} 2^{-(n-k)+1}  f(k) , \quad n \in \Z^+ .
\]
Applying Young's inequality for convolutions (see e.g. \cite[Problem VI.11.10]{DSh58}),
one gets  $(\l^p,\l^q)$-stability for $p \leq q $.
It follows from
$
x (n , 0, \vphi; \fz ) = (2- 2^{-n}) x(0) = (2- 2^{-n}) \vphi^{[0]}
$
that system $x(n+1) = \frac 12 x(n)+ x (0) $ is not UES in $\B^\ga$.% with any $\ga \in \R$.

%\textbf{2)} One can see that the homogeneous system $x(n+1) = \frac 12 x(n)$ defined by the function
%$L_\subd$ is UES in any of $\B^\ga$ with $\ga > 0$, while the system $x(n+1) = \frac 12 x(n)+ x (0) $
%defined by $L$ is not UES in $\B^\ga$, $\ga > 0$.
\end{example}

\vspace{2mm}

The following example shows that in general Theorem \ref{t US} is not valid for
$\ga=0$ and that neither condition (\ref{e LPlinf}) nor (\ref{e tLPlinf})
can be dropped in the US version of Theorem \ref{t ga ind}.

\begin{example}[cf. Example 1 in \cite{BB_JMAA05}] \label{ex US ga=0 repl}
Operators $L(n) \vphi  =  \vphi^{[0]} + \vphi^{[-n]}$
define the system
\begin{eqnarray} \label{e eq exUSga=0}
x(n+1) = x(n)+ x (0) + f(n) , \quad \quad n \in \Z^+ .
\end{eqnarray}

It is easy to see that:
\item[(i)] system (\ref{e eq exUSga=0}) is $(\l^1,\l^\infty)$-stable,
\item[(ii)] condition (\ref{e LPlinf}) is fulfilled for $\ga=0$,
\item[(iii)] but the homogeneous system associated with (\ref{e eq exUSga=0}) 
\emph{is not US} in $\X$ w.r.t. $\B^0$ (and so is not US in $\B^0$).

Indeed, induction shows that
\begin{eqnarray} \label{e x=sumf}
x (n,0,\vz_\X;f) = \sum_{k=0}^{n-1} f(k) , \quad n \in \N .
\end{eqnarray}
So (\ref{e eq exUSga=0}) is $(\l^1,\l^\infty)$-stable.
For $\ga=0$ condition (\ref{e LPlinf}) is fulfilled with $m=0$ since
$\| L (n) \|_{\B^0 \to \X} = 2$.
The system $x(n+1) = x(n)+ x (0)$ is not US in $\X$ since
%\[
$x (n,0,\vphi;\fz) = (n+1) \vphi^{[0]} $. %, \quad n \in \Z^+ .
%\]
\end{example}

\vspace{2mm}

From the next example, one can see that Condition (\ref{e LPl1}) in
Corollary \ref{c UEga=0} cannot be relaxed to condition $\| L (\cdot) \Pr_{[-\infty,-1]} \|_p < \infty$
with $p>1$.

\begin{example}[cf. Example 3.1 in \cite{CKRV97}] \label{ex LPlp}
This is a development of Example \ref{ex US ga=0 repl}.
Let $1<p\leq \infty$, $a (\cdot) \in l^p (\R) \setminus l^1 (\R)$ and
$a(n) \geq 0$ for all $n \in \N$.
Operators $L(n) \vphi  =  \vphi^{[0]} + a(n) \vphi^{[-n]} $ define
%Consider $L:\Z^+ \to \L (\B^{0}, \X)$ defined by
%\begin{eqnarray}
%L(n) \vphi & = & \vphi^{[0]} + a(n) \vphi^{[-n]} . %, \quad \vphi \in \B^0, \quad n \in \Z^+ .
% \label{e L=a}
%\end{eqnarray}
%The function $L$ defines
the system
\begin{eqnarray} \label{e eq exLPlp}
x(n+1) = x(n)+ a(n) x (0) + f(n) . %, \quad \quad n \in \Z^+ .
\end{eqnarray}
One can see that:
\item[(i)] system (\ref{e eq exLPlp}) is $(\l^1,\l^\infty)$-stable,
\item[(ii)] $\| L (\cdot) \Pr_{[-\infty,-1]} \|_p < \infty$,
\item[(iii)] but the homogeneous system associated with (\ref{e eq exLPlp}) 
\emph{is not US in $\X$ w.r.t. $\B^0$}.

Indeed, (\ref{e eq exLPlp}) is $(\l^1,\l^\infty)$-stable since (\ref{e x=sumf}) still holds.
From
\[
| L (n) \Pr_{[-\infty,-1]} \vphi | = | a(n) \vphi^{[-n]} |
\leq a(n) | \vphi |_{\B^0} , %\quad n \in \N,
\]
one can get $\| L (\cdot) \Pr_{[-\infty,-1]} \|_p < \infty$.
The system $x(n+1) = x(n)+ a(n) x (0)$ is not US in $\X$ since
\[
x (n,0,\vphi;\fz) = \left( 1 + \sum_{k=0}^{n-1} a(k) \right) \vphi^{[0]} %, \quad n \in \Z^+ ,
\]
and $\sum_{k=0}^{n-1} a(k) \to +\infty$ as $n \to \infty$.
\end{example}

\vspace{2mm}

The example below shows that the inverse implication in
Proposition \ref{p ga<delta} is not true.

\begin{example} \label{ex ga<delta}
Let $\delta >0$. Consider the system
\begin{eqnarray} \label{e eq exga<delta}
x(n+1) = n e^{- n\delta } x (0) + f(n) . %, \quad \quad n \in \Z^+ .
\end{eqnarray}
Corresponding operators $L(n)$ are defined by
%\begin{eqnarray*}
$L(n) \vphi = n e^{- n\delta } \vphi^{[-n]}$.
%, \quad \quad n \in \Z^+ . \label{e L ga<delta}
%\end{eqnarray*}
Then the homogeneous system associated with (\ref{e eq exga<delta}) is UES in $\X$ w.r.t. $\B^\ga$ 
for all $\ga \in (0, \delta)$ (and so for all $\ga \in (-\infty, \delta)$ due to 
Proposition \ref{p ga<delta}), but \emph{is not US in $\X$ w.r.t.} $\B^\delta$.

Indeed, note that
$x (\cdot +1,0,\vz_\X;f) = f(\cdot)$. Hence system (\ref{e eq exga<delta}) is $(\l^p,\l^q)$-stable for any
$1 \leq p \leq q \leq \infty$. From this and Theorem \ref{t UES} (Theorem \ref{t US}),
we see that system $x(n+1) = n e^{- n\delta } x (0)$ is UES (resp., US) in $\B^\ga$ with $\ga >0$ 
if and only if condition
(\ref{e LPlinf}) is fulfilled. Using the specific form of $L(n)$, one obtains that for $n \geq -m$,
\[
\| L (n) \Pr_{[-\infty,m]} \|_{\B^\ga \to \X} =
\| L (n) \|_{\B^\ga \to \X} =
\sup_{\vphi \neq 0_\B} \frac{n e^{- n\delta } | \vphi^{[-n]} |}{|\vphi|_{\B^\ga}} =
n e^{- n\delta } e^{n\ga} .
\]
Thus, (\ref{e LPlinf}) is fulfilled exactly when $\ga < \delta$.
\end{example}

\section{Discussion and Open Problems \label{s Disc}}

In the present paper we studied the connection of $(\l^p,\l^q)$-stability
and uniform (exponential) stability. % for any $1 \leq p,q \leq \infty$.
The paper comprehensively investigates the topic of
Bohl-Perron type criteria for systems with infinite delay defined on the most popular
%(cf. \cite{CC09,MM05,M97} and references therein)
family $\B^\ga$, $\ga >0$, of fading phase spaces (for the definition of fading phase spaces see
e.g. \cite{MM04} and the discussion in the monograph \cite{HMN91}).
However, there are still relevant open problems.

\textbf{1.}
The choice of an appropriate seminormed or normed phase space $\B$ for the system (\ref{e nh}) 
is determined by the requirement $L (n) \in \L (\B,\X)$ for all $n$.
For the Volterra difference system (\ref{intro1}) and phase spaces $\B^\ga$ this condition takes 
the form \begin{equation}
\label{e Vol boundness cond}
\sum_{k = 0}^{\infty} e^{k \gamma} \| L(n,k) \|_{\X \to \X} < \infty , \qquad n \in \Z^+ .
\end{equation}
If we are interested in the question of uniform (exponential) stability in $\B^\ga$ ($\ga  > 0$), 
we also should take into account the fact that $\sup_{n \geq 0} \| L(n) \|_{\B^\ga \to \X} < \infty$ 
is a necessary condition (see Remark \ref{r |L|<inf} and Proposition \ref{p US to Linf}).
%however this fact is valid for arbitrary phase space $\B$ satisfying Axiom (A) of \cite{M97}).
In the case of Volterra system (\ref{intro1}), condition 
$\sup_{n \geq 0} \| L(n) \|_{\B^\ga \to \X} < \infty$ becomes
\begin{equation} \label{e B0 Vol}
\sup_{n \geq 0} \ \sum_{k = 0}^{\infty} \ e^{k \gamma} \| L(n,k) \|_{\X \to \X} < \infty .
\end{equation}
%(\ref{intro3}) with $l=0$.
On the other hand, having two phase spaces $\B_1$ and $\B_2$ with continuous embedding 
$\| \cdot \|_{\B_2} \leq \| \cdot \|_{\B_1}$, one can see that uniform (exponential) stability in 
$\X$ w.r.t.
$\B_2$ is a stronger property than uniform (exponential) stability in $\X$ w.r.t.
$\B_1$ (cf. Proposition \ref{p ga<delta}).

If (\ref{e Vol boundness cond}) or the stronger condition
(\ref{e B0 Vol})
is violated for all $\ga>0$, but (\ref{e B0 Vol}) holds for $\ga=0$, then it is reasonable to consider 
the question of uniform (exponential) stability in $\X$ w.r.t. $\B^0$. However, the criteria of uniform 
and UE stabilities in fading phase spaces $\B^\ga$ with $\ga >0$
(Theorems \ref{t UES} and \ref{t US}) are \emph{not valid} in the non-fading space $\B^0$ 
(see Examples \ref{ex BB05 ex2} and \ref{ex US ga=0 repl}).

\emph{We were not able to get a comprehensive description of uniform and UE stabilities in $\B^0$ 
in terms of $(\l^p,\l^q)$-stability}. For uniform stability, we obtained a criterion of somewhat different 
type (Theorem \ref{t UEga=0}), but one of its conditions has a complicated nature.

\begin{problem} \label{prob US B0}
Find a complete description of uniform and uniform exponential stabilities in the phase spaces $\B^0$ 
in terms of $(\l^p,\l^q)$-stability and certain boundedness conditions on the coefficients $L(n)$.
\end{problem}

\begin{example}
An example of a Volterra system that is UES in $\X$ w.r.t. $\B^0$ and does not satisfy 
(\ref{e Vol boundness cond}) for all $\ga>0$ is
\begin{equation}
x(n+1) = \sum_{k=-\infty}^{n} \frac{e^{-n}}{(n-k+1)(n-k+2)} x(k) , \quad n\geq 0.
\end{equation}
Indeed, it is easy to see by induction that 
$| x (\tau+k,\tau, \vphi; \fz) | \leq e^{-(\tau+k-1)} | \vphi |_{\B^0} $ 
for all $\tau \geq 0$ and $k \geq 1$.
\end{example}

Note that in \cite{CKRV00} exponential stability is understood in the sense of exponentially 
decaying estimates on the fundamental matrix.  This stability property does not depend on 
the choice of the phase space and, for the particular case of Volterra systems of 
the form (\ref{e Volunb}), is stronger than UES stability in $B^\ga$ for any $\ga >0 $. 
The idea of \cite{CKRV00} to consider  $(\l^p,\l^q)$-stability with weighted $\l^p$-spaces
may provide an approach to Problem \ref{prob US B0}.

In Corollary \ref{c UEga=0} we give two sufficient conditions of Bohl-Perron type for uniform stability 
in $\B^0$. One can see that, in these conditions, the assumption on $L(n)$ is connected with 
the parameter $p$ (which equals $1$ or $\infty$) in $(\l^p,\l^\infty)$-stability.
The question is whether the interpolation to intermediate values of $p \in (1,\infty)$ also provides
sufficient conditions.

\begin{problem} \label{prob interp}
More precisely, for $p \in (1,\infty)$, prove or disprove that system (\ref{e h}) is US in $\B^{0}$ if  
 (\ref{e nh}) is $(\l^p,\l^\infty)$-stable and
\begin{equation*} \label{open2}
\text{there exists } \ m \in \Z^- \ \ \text{such that} \quad
\| L (\cdot) \Pr_{[-\infty,m]} \|_p^p :=
\sum_{n=0}^{\infty} \| L (n) \Pr_{[-\infty,m]} \|_{\B^{0} \to \X}^p < \infty \ .
\end{equation*}
\end{problem}

\vspace{2mm}

\textbf{2.} We used essentially the fact that the phase spaces $\B^\ga$ with $\ga >0$ are fading.
Our results can be extended on wider class of phase spaces of $\l^r$ type defined by
\begin{equation*}
\B^{r,\ga} :=\{ \vphi = \col ( \vphi^{[m]} )_{m=-\infty}^{0}  \ :
\ | \vphi |_{\B^{r,\ga}} := 
\left( \sum_{m=-\infty}^{0} |e^{\ga m} \vphi^{[m]} |^r  \right)^{1/r} \!\! < \infty \}, 
\quad 1 \leq r < \infty .
\end{equation*}
If $\ga >0$, the phase spaces $\B^{r,\ga}$ are fading. We show below that our method works for
this class.

\begin{theorem}
Let $1\leq r < \infty$ and $\ga >0$. Then Theorems \ref{t UES} and \ref{t US} are valid with 
the phase space $\B=\B^{r,\ga}$ instead of $\B^\ga$.
\end{theorem}

\begin{proof}
The general scheme of the proof of Theorems \ref{t UES} and \ref{t US} works with minor changes in
(\ref{e |phi|P0}), (\ref{e |D|}), (\ref{e ||<||linf}), (\ref{e supLP}), and more essential changes
in Proposition \ref{p US} and
in the proofs of Propositions \ref{p UES} and \ref{p hlp}.
We explain here only essential changes.

\textbf{Changes in the proof of Proposition \ref{p hlp}.}
Estimate (\ref{e est h(n)}) can be adjusted to the phase space $\B^{r,\ga}$ with the use of 
Minkowski's inequality:
\begin{eqnarray*}
| \h (n) |_{\B^{r,\ga}} & = & \left( \sum_{m \in \Z^-}
\left| \sum_{k=\max \{ 0, m+n \} }^{n-1} e^{(-n+k+1)\ga} 
\left[ e^{(m+n-k-1) \ga}  g^{[m+n-k-1]} (k) \right]  \right|^r \right)^{1/r}
\leq \\ & \leq &
 \sum_{k= 0}^{n-1} e^{(-n+k+1)\ga} | g (k) |_{\B^{r,\ga}} =
  (\e * \g_r) (n) ,
\end{eqnarray*}
where $\e (\cdot)$ is given by (\ref{e def g e}) and $\g_r (\cdot) := | g (\cdot) |_{\B^{r,\ga}} $.
The rest of the proof is the same.

\textbf{Changes in the proof of Proposition \ref{p UES}.} We use
$| \vphi |_{\B^{r,\ga}} \leq
|\Pr_{\{0\}} \vphi |_{\B^{r,\ga}} + | (\I - \Pr_{\{0\}}) \vphi |_{\B^{r,\ga}} $ instead of
(\ref{e |phi|P0}). Then induction easily produce the inequality
\[
| y(n,\tau, \vphi; \fz) |_{\B^{r,\ga}}
\leq (n-\tau+1) K e^{- \nu_1  (n-\tau) } | \vphi |_{\B^{r,\ga}} .
\]
Replacing $\nu_1$ with any $\nu_2 \in (0,\nu_1)$ and changing $(n-\tau+1) K$ to large enough 
constant $K_1$, one gets
UE stability in $\B^{r,\ga}$.

\textbf{For phase spaces $\B^{r,\ga}$ Proposition \ref{p US} is valid only when $\ga>0$, the proof
for this case requires the following changes.} We again use $| \vphi |_{\B^{r,\ga}} \leq
|\Pr_{\{0\}} \vphi |_{\B^{r,\ga}} + | (\I - \Pr_{\{0\}}) \vphi |_{\B^{r,\ga}} $ instead of
(\ref{e |phi|P0}) and get by induction the inequality
\[
| y(n,\tau, \vphi; \fz) |_{\B^{r,\ga}}
\leq K \left( \sum_{j=0}^{n-\tau} e^{-j\ga}\right)  | \vphi |_{\B^{r,\ga}} .
\]
This implies uniform stability in $\B^{r,\ga}$ with $K_1=K (1-e^{-\ga})^{-1}$.
\end{proof}

The non-fading memory spaces $\B^{r,0}$
%(and $\B^0$, which can be identified with $\B^{\infty,0}$)
appear naturally in the theory of Volterra difference systems (\ref{intro1}) when
the coefficients $L(n,k)$ satisfy condition of the $\l^{r'}$ type:
\[
\sum_{k =0}^{+\infty} \| L(n,k) \|_{\X \to \X}^{r'} < \infty \qquad 
\left(\text{or, for } r=1, \ \sup_{k \geq 0} \| L(n,k) \|_{\X \to \X}  \right) ,  \qquad n \in \Z^+
\]
(it is assumed that $1/r+1/r'=1$).
Note that $\B^{r,0}$ compose a scale between $\bigcup\limits_{\ga < 0} \B^\ga$ and $\B^0$ in 
the sense that the embedding
\begin{equation} \label{e emL}
 \B^\ga \subseteq \B^{r,0} \subseteq \B^0 , \quad 
 | \vphi |_{\B^{0}} \leq | \vphi |_{\B^{r,0}} \leq (1 - e^{\ga r})^{-1/r} | \vphi |_{\B^{\ga}}, 
 \quad %\text{for} \quad
 \ga < 0, \ 1 \leq r < \infty ,
\end{equation}
holds as well as the usual $\l^r$ embedding of spaces $\B^{r,0}$.

\begin{problem}
Find Bohl-Perron type criterions for phase spaces $\B^{r,0}$. This question widens frames of Problem 
\ref{prob US B0}.
\end{problem}

\textbf{3.}  Bohl-Perron type theorems were used in \cite[Sec. 5]{BB_JMAA05} to derive explicit
(i.e., given in the terms of coefficients) exponential stability tests for systems with bounded delay.
It is interesting to apply the method of \cite[Sec. 5]{BB_JMAA05} to Volterra difference systems with 
unbounded and infinite delay. This is a part of the following general problem.

\begin{problem}
Find explicit  tests of exponential stability (complementing the known ones)
for Volterra difference systems with unbounded or infinite delay.
%Can explicit exponentially stability tests (complementing the known ones) be deduced
%for the Volterra difference equation (\ref{intro1}) with infinite delay
%$$
%x(n+1)=\sum_{k=0}^n B(n,k)x(k)
%$$
%and the equation with an infinite delay
%$$
%x(n+1)=\sum_{k=-\infty}^n B(n,k)x(k)
%$$
%under some convergence and/or boundedness conditions on
%$\sum_{k=-\infty}^n B(n,k)$ and
%$L(n,k)$?
\end{problem}

One of the motivations for this question is that exponential stability (and, more generally, dichotomy) 
is used in the study of bounded solutions to nonlinear perturbations of Volterra difference equations 
(see e.g. \cite{CC09,M97}).

On the other hand, there is not much literature devoted to explicit conditions of exponential stability
in the cases of unbounded and infinite delay.
For systems of convolution type
\begin{equation} \label{e Vconv}
x(n+1)=\sum_{k=0}^{n} B(n-k) x(k)
\end{equation}
some sufficient conditions can be derived from known results on asymptotical stability 
(see e.g. \cite{E05,KC-VT-M03}) with the use of \cite[Theorem 5]{EM96}. (\cite[Theorem 5]{EM96} 
gives a necessary and sufficient condition of exponential stability of (\ref{e Vconv}) under 
the assumption that (\ref{e Vconv}) is asymptotically stable.)

\centerline{\bf Acknowledgment}

The authors were partially supported by the Pacific Institute for
Mathematical Sciences and by NSERC research grant.

\end{document}